\documentclass[12pt]{article}

\usepackage{xcolor}
\usepackage{amssymb}
\usepackage{mathrsfs}

\usepackage{common}

\title{Classification over a predicate - the general case Part I - structure theory}

\author{Saharon Shelah\footnote{Publication no. 322 on Shelah's list of publications. Shelah thanks the Israel Science Foundation
   (Grants 1053/11, 1838/19), and the
   European Research Council (Grant 338821)
   for partial support of this research.} and  Alexander Usvyatsov\footnote{Research partially supported by Marie Sklodowska Curie CIG 321915 "ModStabBan". Usvyatsov thanks the Austrian Science Foundation (FWF), projects P33895 and  P33420, and the Portuguese Science Foundation (FCT), Development grant IF/01726/2012, for their partial support at different stages of this research.
    }}

\newtheorem{theorem}{Theorem}[section]
\newtheorem{definition}[theorem]{Definition}
\newtheorem{example}[theorem]{Example}
\newtheorem{lem}[theorem]{Lemma}
\newtheorem{obs}[theorem]{Observation}
\newtheorem{co}[theorem]{Corollary}

\newtheorem{hyp}[theorem]{Hypothesis}
\newtheorem{remark}[theorem]{Remark}
\newtheorem{pr}[theorem]{Proposition}
\newtheorem{no}[theorem]{Notation}

\newtheorem{conv}[theorem]{Convention}

\newtheorem{qst}[theorem]{Question}
\newtheorem{ft}[theorem]{Fact}
\renewenvironment{proof}{\noindent {\em Proof:}}{\hspace*{1cm}
        \hspace*{\fill}$\rule{1.2ex}{1.4ex}$\medskip} 
\newenvironment{re}{\begin{remark}\rm}{\end{remark}}
\newenvironment{ex}{\begin{example}\rm}{\end{example}} 
\newenvironment{de}{\begin{definition}\rm}{\end{definition}}
\newenvironment{fact}{\begin{ft}\rm}{\end{ft}}

\newcommand{\red}[1]{{\color{black}{#1}}}
\newcommand{\green}[1]{{\color{black}{#1}}}

\date{}

\begin{document}


\maketitle

\abstract{We begin a systematic development of structure theory for a first order theory, which is stable over a monadic predicate. We show that stability over a predicate implies quantifier free definability of types over stable sets, introduce an independence notion and explore its properties, prove stable amalgamation results, and show that every type over a model, orthogonal to the predicate, is generically stable.}

\section{Introduction}

The classical Classification Theory \cite{Sh:c} deals with a first order theory
$T$, how complicated its models can be, and to  which extent they can be
characterized by cardinal invariants.

For algebraically closed fields, or for divisible abelian groups, there is
such a structure theory. In general, there is a division into 
theories that have a structure 
and such for which we have a
non-structure theorem.

Consider now vector spaces over a field. Do we have a structure theory? It
depends on how you ask the question: We have a structure theory for the
vector space over the field, but  not necessarily for the field.

\underline{Problem}: Classify pairs $(T,P)$ where $T$ is a first order
theory and $P$ is a monadic predicate. We want to know how much
$M\models T$ is determined by $M|{P^M}$. So in case $T$ is the two-sorted theory of a vector space over a field, and $P$ is the predicate for the field, we ask how much we can know about a vector space over a field $F$ once $F$ is flixed (obviously, we know quite a bit, especially if the cardinality of the vector space is fixed). 

Although at the first glance the problem above may appear close to classical (first order) model theory, this context actually exhibits behavior which is more similar to that of some non-elementary classes (classes of models of a sentence in an infinitary logic, or abstract elementary classes). See e.g. Hart and Shelah \cite{HaSh323}. An intuitive  ``explanation'' for this is that fixing $P$ is similar to insisting on omitting a certain type (the type enlarging $P$), which immediately puts one in a non-elementary context. 

Much work has already been done on classification theory over $P$. Relative categoricity for particular theories was investigated by e.g. Hodges et al \cite{Hod-cat1,Hod-cat2,Hod-cat3}; countable categoricity over $P$ was studied by Pillay \cite{Pil-cat}. Pillay and the first author laid  the foundations for the study of stability and related properties in this context in \cite{PiSh130}. Here we are going to continue their investigation and build upon their results.

In \cite{Sh234} the first author proved an analogue of Morley's Theorem over $P$ under the assumption of ``no two-cardinal models'', which means that for all $M\models T$, $|P^M| = |M|$. However, this assumption is very  strong. Even for the the case of uncountable categoricity (a notion that  we discuss below in Definition \ref{dfn:cat}), it would be nice to be able to prove an analogue of Morley's Theorem without making this assumotion it \emph{a priori}.  Furthermore, the ``no two cardinal model'' assumption does not hold in many natural examples that 
should  ``morally'' be quite tame. For example, the general theory developed in \cite{Sh234} does not even cover  the example of a vector space over a field, already mentioned briefly above  (Example \ref{ex:vs}). Apparently, from the point of view of model theory, this example is not that easy (but is quite instructive). 


A much more interesting example that we hope will eventually be included in our treatment is the theory of so-called ``Zilber's  field'', more precisely, the theory of exponentially closed fields of characteristic 0 \cite{Zil-field2,KiZil, Hen-thesis}. In our context,  one considers the theory of an exponentially closed field $(F,e^z)$ over $P$, which is the kernel of the exponential function $e^z$. For example, it follows from the results in \cite{KiZil} that this theory is stable over $P$. 

Another natural example to consider in this setting is the theory of algebraically closed field of characteristic 0 with a generic automorphism $\sigma$, where $P$ is the fixed field of $\sigma$. This theory has  been studied by Chatzidakis in a recent preprint  \cite{Chatzidakis2020RemarksAT}. This example is different than the ones mentioned above, since in this case $P$ is model theoretically ``tame'' (it is pseudo-finite), and the full theory $T$ is also very well understood (it is simple, has QE, etc). Nevertheless, considering it in this new framework, offers further insight. 

Some results in \cite{Chatzidakis2020RemarksAT} can be derived directly from the more general analysis obtained in our paper. For example, Chatzidakis proves (Theorem 3.14 in \cite{Chatzidakis2020RemarksAT}) the existence of a $\ka$-prime  $\ka$-atomic model over an algebraically closed difference field of characteristic $0$ with a $\ka$-saturated and pseudo-finite $P$-part (for $\ka$ uncountable or $\aleph_\eps$). We prove an analogous result in our general context (see Theorem \ref{th:primary} and  Corollary \ref{co:primary} thereafter) for the case that the $P$-part is saturated (of cardinality $\ka$). So in particular, the case $\ka = \aleph_\eps$ is not covered. A more nuanced analysis, which would imply stronger results closer to Theorem 3.14 in \cite{Chatzidakis2020RemarksAT} goes beyond the scope of this article. This will be investigated in a future work. 

\medskip

Let us make the above discussion a bit more precise. 

One rough measure of complexity of a theory is (as usual) the number of its non-isomorphic models in different cardinalities:

\begin{de}
\begin{itemize}
\item[(i)] $I(\lambda ,N)=$ cardinality of $\{M/_{\cong _N} :
M\models T, M|_P=N, |M|=\lambda\}$, where $\cong _N$ means
isomorphic over $N$ and  $M/_{\cong _N}$ denotes the isomorphism class.
\item[(ii)] $I(\lambda ,\mu)= \sup \{I(\lambda, N) :|N|=\mu\}$.
\item[(iii)] $I(\lambda)=:I(\lambda,\lambda )$.
\end{itemize}
\end{de}

\begin{ex}\label{ex:vs} (Back to vector spaces over a field). Let $T$ be the theory of two-sorted models $(V,F)$ where $V$
is a vector space over the field $F$, and let $P$ be a predicate for the
field $F$. If the cardinality of $F$ is $\aleph_\alpha$, then
$I(\aleph _\alpha, F)=|\omega +\alpha|$ and for $\beta >\alpha$, $I(\aleph
_\beta,{\color{black}F})=1$.
\end{ex}


So we want to divide the pairs $(T,P)$ according to how much freedom we
have to determine $M$ knowing $M|P$.

\begin{de}\label{dfn:cat}
\begin{itemize}
\item[(i)] $(T,P)$ is \emph{categorical in $(\lambda_1,\lambda_2)$} when the following
holds: If $M_1|P=M_2|P, M_1$ and $M_2$ are models of
$T$ and $|M_l|=\lambda_1$, $| P^{M_l}|=\lambda_2$ for $l=1,2$, then $M_1$ and
$M_2$ are isomorphic over $P^{M_l}$.
\item[(ii)] We write \emph{categorical in $\lambda$} instead of $(\lam,\lam)$.
\item[(iii)] We say \emph{totally categorical} if (ii) holds for all $\lambda$.
\end{itemize}
\end{de}

For our purpose, categorical pairs are the simplest. However, we will deal here with a more general context of \emph{stability} over $P$. For example, the class of vector spaces over a field is not categorical in $\lam$ the sense of the definition above; still, it is \emph{almost} categorical, and it would obviously be desirable to develop a general theory that covers this example. 

%
\medskip

Recall that much of the work in classification theory follows the following general recipe. First, we assume that the theory $T$ (or, more generally, the class of models under investigation) has a particular ``bad'' model theoretic property: e.g., is unstable. Under this assumption, we prove a non-structure theorem: e.g., $T$ has many non-isomorphic models of some big enough cardinality $\lam$ (normally, in many, if not all, such cardinalities). Since we are ultimately interested in the ``good case'' -- for example, if we are trying to prove an analogue of Morley's Categoricity Theorem, we only care about theories (or classes) with few models -- we may assume from now on that $T$ falls on the ``right'' side of the dividing line: e.g., is stable. This way we can use good properties of stability in order to investigate properties of our class further. 

As perhaps should be clear from the title, in this paper we focus on the development of \emph{structure theory}. In particular, we do not prove any new non-structure results here. However, we do follow ``recipe'' described above. That is, we recall (mostly from \cite{PiSh130} and \cite{Sh234}) that certain ``bad'' properties (e.g. instability) imply non-structure, by which we normally mean many non-isomorphic models over $P$ -- sometimes only in a forcing extension of the universe, or under mild set-theoretic assumptions. Non-structure theorems become quite hard in this context, so  at the moment we are happy with just consistency results. Since we are ultimately interested in absolute properties (e.g., stability, categoricity), there does not seem to be much loss in this approach (although non-structure results in $ZFC$ are definitely on our to-do list). 
Then we restrict our attention to theories on the ``good side'' of all the dividing lines, and focus on proving structure results for this case. There is, of course, also a need for better and stronger non-structure theorems, that would ``justify''  restricting onelsef to even nicer contexts; this is a topic for future work.

\bigskip

This paper is organized as follows.

 In Section 2 we recall  non-structure results from \cite{PiSh130} and make our most basic working assumptions, e..g., every type over $P$ is definable. 
 
 Section 3 is devoted to some of the less obvious consequences of the assumptions mentioned above (already discussed to a large extent in \cite{PiSh130}). In particular, we explain why we may assume that $T$ has quantifier elimination. 

In Section 4 we revisit some basic stability theory over $P$ originally developed in \cite{PiSh130} and \cite{Sh234} and take it further. We introduce some of the major players necessary for analyzing models over their $P$-part:  complete sets and ``good'' types, which we call *-types; these are types ``orthogonal to $P$'', that is, types, realizing which do not increase the $P$-part. Based on these notions, we define the relevant notion of stability: a set $A$ is called stable (over $P$) if there are ``few'' *-types over all sets elementarily equivalent to $A$. This concept and many of its basic properties already appear  in the previous works mentioned above. However, it is our goal to make this paper reasonably self-contained, so we include all the definitions as well as details of proofs. 

Section 5 revisits a notion of rank  (already discussed in \cite{PiSh130, Sh234}) for *-types that captures stability over $P$. 

Section 6 is devoted to the existence of ``small''  (atomic, primary) saturated models over stable sets with a saturated $P$-part. 

In Section 7 we make an additional working assumption: every model of $T$ is stable over $P$ (this hypothesis is justified by a non-structure result proved in \cite{Sh234}). We then use 
stability of models together with the existence of primary models (the results of section 6) in order to obtain quantifier-free internal definitions for *-types. This is, in a sense, the main technical result of this paper, on which the rest of our analysis is built.  In particular, we immediately conclude that *-types over stable sets are definable (in $\cC$). 

In section 8 we define stationarization -- a version of non-forking independence appropriate for our context -- and examine its basic properties. 

Finally, in  Section 9 we prove the main structure results of this article. In particular, we establish stable amalgamation for models followed by a similar result for types over stable sets, show symmetry of independence over models, and draw several conclusions. 

One basic corollary of our analysis is that if every model is stable over $P$
(as mentioned earlier, a non-structure result from \cite{Sh234} implies that this is a natural assumption in our context; in addition, it  holds in many interesting examples), then every $*$-type over a model  is generically stable. In particular, we conclude that over models, independence derived from stationarization coincides with non-forking. 

\subsection*{Acknowledgements} The authors are grateful to Mariana Vicaria for a careful reading, and for numerous  comments, questions, suggestions, and corrections. We also thank Zo\'{e} Chatzidakis for helpful comments, clarifications, and instructive questions, which led to improving the readability and clarity of this paper. 

The second author  thanks Jakob Kellner for several conversations, motivating the authors to expand the background and the introductory sections, hopefully making them more clear and accessible to non-experts, and to eventually finalize the paper.  Usvyatsov is also thankful to Alf Onshuus for inspiring conversations at early stages of the writing of this paper, and to Jonathan Kirby for bringing the paper \cite{KiZil} to the authors' attention. 

\section{The Gross Non-structure Cases}

\begin{conv}
Let $T$ be a complete  first order theory, $P$ a monadic predicate in
its vocabulary. 
\end{conv}

\noindent
Let $\mathcal{C} $ be the monster model of $T$. From now on, we assume that all models of $T$ are elementary submodels of $\mathcal{C}$, and all sets are subsets of $\mathcal{C}$. \medskip
 
For $M \models T$, we denote by  $M|_P$ the set $P^M$ viewed as a substructure of $M$. Similarly, for a subset $A \subseteq M$, we denote by $M|_A$ the substructure of $M$ with universe $A$. We write $A\equiv B$ if $Th({\cal C}|_A)=Th({\cal
C}|_B)$.

We also denote $T^P = Th(\cC|_P)$. For a set $A$, we denote $P^A = A \cap P^\cC$. 

%

\bigskip

Our first dividing line concerns the connection between subsets of $P^\cC$ that are 0-definable externally (in $\cC$) and internally (in $\cC|_P$, that is, in $P^\cC$ viewed as a substructure of $\cC$). One direction of this correspondence is  straightforward:

\green{
	\begin{remark}
	\label{rmk:externaltointernal}
		For every formula $\theta(\x)$ there exists a formula $\theta^P(\x)$ such that for every $\b \in P$ we have 
		$\cC|_P \models \theta(\b)$ if and only $\cC\models \theta^P(\b)$. Moreover, for every  $M \models T$ we have
		$M|_{P} \models  \theta(\b)$ if and only $M \models \theta^P(\b)$
	\end{remark}
	\begin{proof}
		By induction on $\theta$, by replacing quantifiers in $\theta$ by  quantifiers restricted to $P$ (e.g., $\exists y \in P$). 
	\end{proof}

\begin{co}\label{co:fromTPtoT}
	Let $A$ be a set. 
	\begin{enumerate}
	\item
	If $A$ is a model (so $A \elem \cC$), then  $(A|_P = ) \cC|_{P^A} \elem  \cC|_{P}$.
	\item
	Every (partial) $T^P$-type $q(\x)$ over $P^A$ is equivalent to a $T$-type $q^P(\x)$ (over $P^A$) with \, $[\x \subseteq P] \in q^P$. 
	
	In particular, if $A$  is a $\lam$-saturated model, then $A|_P = \cC|_{P^A}$ is a $\lam$-saturated model of $T^P$.
	\end{enumerate}
\end{co}
\begin{proof}
\begin{enumerate}
\item Assume $\cC|_P \models \theta(\b)$ with $\b \in P^A$. Then (see Remark \ref{rmk:externaltointernal}) 
$\cC \models \theta^P(\b)$, hence $A \models \theta^P(\b)$. Again, by Remark \ref{rmk:externaltointernal}, 
$A|_P\models \theta(\b)$. 
\item Let $\theta(\x,\b)$ be a formula such that $A|_P \models \exists \x \theta(\x,\b)$. Then 
$$A \models \exists \x\in P\; \theta^P(\x,\b)$$ 

Hence given a $T^P$-type $q(\x)$ over $P^A$, the collection $q^P = \set{\theta^P(\x)\colon \theta(\x)\in q}$ is indeed a $T$-type, and the rest should be clear. 

\end{enumerate}

\end{proof}

\medskip

On the other hand, it is not clear (and, in general, not true) that any externally definable subset of $P^\cC$ is definable internally in $\cC|_P$. 

}

\noindent

\medskip

\begin{qst}\label{qst:1}
If $M\models T$, and $\psi (\red{\bar x})$ is a relation
on $P^M$ which is first order definable in $M$, is $\psi$ definable in
$M|_{P}$, possibly with parameters?
\end{qst}


If the answer to this question is ``no'', then, as shown in
\cite{PiSh130}, for every $\lambda\geq |T|$, $I(\lambda)\geq
Ded(\lambda)$, where $Ded(\lambda)=\sup \{ |I|:I$ is a linear order
with dense subset of power $\lambda\}$. The proof relies on earlier papers by
Chang, Makkai, Reyes and Shelah.

Following the ``general recipe of Classification Theory'' outlined in the introduction, we therefore assume that the answer to the Question \ref{qst:1} is
``yes''.

\bigskip

Furthermore we expand $T$ by the necessary individual constants 
(maybe working in
${\cal C}^{eq}$) in order to assume that all such relations are parameter-free
definable in ${\cal C}|_{P}$. Note that we only have to add elements in $dcl(P^{\cal C})$
inside ${\cal C}^{eq}$. 

\smallskip

Specifically, 
if $\cC \models (\exists\bar y \in P)(\forall\bar x \in P)
(\psi (\bar x)=\theta^P(\bar x,\bar y))$, 
%
%
\red{we define an equivalence relation on tuples of sort $\y$ as follows:

\[ \y_1 \equiv \y_2 \iff P(\y_1)\land P(\y_2) \land \forall \x \left[ \theta^P(\x,y_1) \leftrightarrow \theta^P(\x,y_2)\right] \]

\noindent
and expand $\cC^{eq}$ with names for the relevant equivalence classes; that is, for each $\psi(\x)$ as above, expand 
the language by a constant $[\y]_\psi$ for the class of a tuple $\b$ for which $\theta^P(\x,\b)$ is equivalent to $\psi(\x)$. Clearly, this does not increase the cardinality of the language.

As a result, given a formula $\psi(\x)$ defining a relation on $P$ and a formula $\theta(\x,\b)$ defining this relation in $\cC|_P$, we have that e.g. the formula $\exists \y \left( [\y]=[\y]_\psi \right) \land \theta(\x, \y)$ defines the same relation without parameters. Note that we add the new constants to $P$ and to the language of $T^P$. 


}

\begin{re} See \cite{PiSh130} for the number of such expansions for models
of $T$. If the new theory is $T^+$, note that $I_T(\lambda ,
N)=\sup\{I_{T^+}(\lambda, N^+):N^+$ an expansion of $N$ as described
above$\}$. In particular $I_T(\lambda,\mu)^{|T|}\geq
I_T(\lambda,\mu)\geq I_{T^+}(\lambda,\mu)$, so our non-structure results
are not affected.
\end{re}

\noindent
\begin{qst}\label{qst:2}
For $\psi=\psi(\bar x,\bar y)$, $\bar c\in {\cal C}$,
is $tp_\psi(\bar c/P^{\cal C})$ definable?
\end{qst}

\smallskip

Recall that $tp_\psi(\bar c/A)=\{\psi(\bar x,\bar a):\bar a\in  
A, \; \cC \models \psi(\bar c,\bar a)\}$ and $p=tp_\psi (\bar c/A)$ is definable if it
is definable by some $\theta(\bar y,\bar d)$ with $\bar d\in A$, which means that for
every $\bar a\in A$ we have  $\psi(\bar x,\bar a)\in p\Longleftrightarrow
\, \cC \models\theta(\bar a,\bar d)$. We restrict ourselves to $\psi $-types in order to be able to use
compactness arguments.

\smallskip

Again, it is shown in  \cite{PiSh130} that if the answer to this question is ``no'', then $I(\lambda)\geq
Ded(\lambda)$. The two questions are closely related (the second one is a version of the first one with external parameters), and the proofs are similar, and, in fact, both work
for pseudo-elementary classes.

\bigskip

In conclusion, for the rest of the paper we make the following assumptions:

\begin{hyp}\label{asm:1} (\emph{\underline{Hypothesis 1})}

  $T$ is a  complete first order theory 
, 
  $P$ a monadic predicate in the language of $T$, $\mathcal{C} $ is the monster model of $T$ such that  
  
\begin{enumerate}
 \item  Subsets $P^\mathcal{C}$  that are 0-definable (in $\mathcal{C}$),  are already  0-definable in
$\mathcal{C}|_{P}$. 
 \item Every type over $P^\mathcal{C}$ is definable. 
\end{enumerate}
\end{hyp}

%
%
%


Later we shall add an additional clause to this hypothesis regarding quantifier elimination of $T$; see Hypothesis \ref{hyp:1.5}.

\medskip
For simplicity we also make the set theoretic assumption that for
arbitrarily large $\lambda$, we have $\lambda ^{<\lambda}=\lambda$ (in particular, $T$ has arbitrarily large saturated models). Note that for any conclusion we draw
which says something about every \red{$\lambda$}, this hypothesis can be eliminated.


\section{Internal and external definability}



In this section we observe several basic consequences of Hypothesis 1. 

\medskip

First we reformulate the assumption  that every type over $P^\cC$ is definable in the following obviously equivalent  way: every externally definable subset of $P^\cC$ is also internally definable in $\cC|_P$ (with parameters in $P^\cC$). 

\begin{obs}
\label{obs:def_stablyembedded}
\begin{enumerate}
\item
	Let $\ph(\x,\a)$ define (in $\cC$) a subset of $P^\cC$.  Then there exists a formula $\hat\theta(\x,\b)$ with $\b \in P^\cC$ that defines the same set. Moreover, there exists a formula $\theta(\x,\b)$ that defines the same set in $\cC|_P$. 
\item
	Furthermore, for any $A \subseteq P^\cC$ such that $\tp(\a/P^\cC)$ is definable over $A$, $\theta(\x,\b)$ in the previous clause can be chosen so that $\b \in P^A$.
\end{enumerate}
\end{obs}
\begin{proof}
	By Hypothesis \ref{asm:1}(ii),  $\tp(a/P^\cC)$ is definable. Hence there is $\hat\theta(\x,\b)$ with $\b \in P^\cC$ such that $\cC \models \forall \x \left( \hat\theta(\x,\b) \longleftrightarrow \ph(\x,\a) \right)$. By Hypothesis \ref{asm:1}(i), there exists $\theta(\x,\y)$ such that for every $\c,\d \in P^\cC$ we have $\cC \models \hat\theta(\c,\d)$ if and only if 
	$\cC|_P \models \theta(\c,\d)$. Clearly $\theta(\x,\b)$ is as required (in both clauses of the Observation).
\end{proof}

\smallskip


\begin{co}\label{co:P-type}
\begin{enumerate}
\item
	 Let $p(\x)$ be a (partial) type over $\cC$ such that $P(\x) \in p$. Then $p$ is equivalent to a $T^P$-type $p'$ over 
	$P^\cC$.
%
\item
	Let $A$ be a set. Assume that for every $\a \in A$ the type $\tp(\a/P^\cC)$ is definable over $P^A = P^\cC\cap A$. Let $p$ be a (partial) type over $A$ with $P(x) \in p$. Then $p$  is equivalent to a ($T$-)type $\hat p$ over $P^A$, and to a $T^P$-type $p'$ over $P^A$.
\end{enumerate}
\end{co}
%


The following corollary again follows  immediately, but is  very useful. 

\begin{obs}\label{obs:P-sat}
	Let $A$ be a set such that for every $\a \in A$ the type $\tp(\a/P^\cC)$ is definable over $P^A$. Let $\lam$ be a 
	cardinal.
	\begin{enumerate}
	\item
	Assume that $\cC|_{P^A}$ is a $\lam$-saturated (or just $\lam$-compact) model of $T^P$. Then every $T$-type 
	of size $<\lam$ over $P^A$ which is finitely satisfiable in $P^A$ is realized in $P^A$.  
	
	In fact, it is enough to assume that any $T^P$-type of size $<\lam$ over $P^A$ which is finitely satisfiable in $P^A$ is realized in $P^A$ (so $\cC|_{P^A}$ does not need to be itself a model of $T^P$). 
	
	\item
		If $P^A$ satisfies the conclusion of the clause (i), then every type $p$ of size $<\lam$ over $A$ with $P(x) \in p$
		is realized in $P^A$. 
		
	\end{enumerate}
\end{obs}
\begin{proof}
	For clause (i), let $p$ is a $T$-type of size $<\lam$ over $P^A$ finitely satisfiable in $P^A$. By Corollary \ref{co:P-type}(ii), it is equivalent to a $T^P$-type over $P^A$, and so is realized in $P^A$. For clause (ii), note that, again by  Corollary \ref{co:P-type}(ii), such a type $p$ is equivalent to a type of size $<\lam$ over $P^A$ (and clearly it is finitely satisfiable in $P^A$). 
\end{proof}

Let us formulate a simple version of the last Observation, which will be particularly useful:

\begin{co}\label{co:P-sat}
	Let $A$ be a set such that for every $\a \in A$ the type $\tp(\a/P^\cC)$ is definable over $P^A$. Let $\lam$ be a 
	cardinal, and assume that $\cC|_{P^A}$ is a $\lam$-saturated model of $T^P$. Then every $T$-type $p$ of size $<\lam$ over $A$ with $P(x) \in p$
		is realized in $P^A$. 

\end{co}

\begin{re}\label{re:compactvssat}
	In the rest of the paper, we will prove  claims about sets and models under the assumption that their $P$-part is 
	$\lam$-saturated. As a matter of fact, all we'll need in these claims is the conclusion of Observation \ref{obs:P-sat} (i). So
	in particular assuming that $P^A$ is a ``$\lam$-compact subset'' of $\cC$ (every $T$-type of size $<\lam$ which is finitely satisfiable in $P^A$ is realized in $P^A$) is enough. But we will not focus on this.
\end{re}

\medskip

A perhaps less obvious consequence of Hypothesis 1 is that $T$ can be Morleyrized without much additional cost. 

In classical Classification Theory (classifying first order theories), assuming that $T$ has quantifier elimination often comes without much cost (e.g., by Morleyzation). However, in our case, this assumption may appear less harmless. Indeed, in general, if $\cC$ is expanded with new predicates, then  so is $\cC|_P$, potentially enhancing the theory $T^P$, hence  changing the original classification problem. Fortunately, since Morleyzation of $T$ does not add new $0$-definable sets to $\cC$, if Hypothesis 1 is true, no new $0$-definable sets are added to $\cC|_P$ either. In other words: 

\begin{obs}\label{obs:QE_harmless}
	Let $T'$ is the Morleyzation of $T$, and let $\cC'$ be the expanded monster model (of $T'$). Let $T'^P$ be the theory 
	of $\cC'|_P$. Then $T'^P$ is a trivial expansion of the Morleyzation of $T^P$ (that is, it is an expansion that adds no new $0$-definable sets). 
\end{obs}
\begin{proof}
	On the one hand, if $\theta(\x)$ be a $T_P$-formula,  let $\theta^P(\x)$ be the formula that defines the same subset of $P^\cC$ in $\cC$ (as in Remark \ref{rmk:externaltointernal}). Let $R(\x)$ be a $T'$-predicate equivalent (modulo $T'$) to $\theta^P(\x)$; then it also defines the same subset of $\cC'_P$, hence is equivalent to $\theta(\x)$ modulo $T'^P$. In conclusion, $T'_P$ \emph{expands} the Morleyzation of $T^P$ (this part is, of course, always true). 
	
	On the other hand, let $R(\x)$ be a new predicate in $T'^P$. Then it is equivalent modulo $T'$ to a $T$-formula 
	\[
		\x \subseteq P \bigwedge \ph(\x)
	\]
	
	By Hypothesis 1, there is a $T^P$-formula $\theta(\x)$ (without parameters) defining the same subset of $P^\cC$. 
	Therefore any $T'_P$ formula $\theta'(\x)$ is equivalent to a $T_P$ formula, as required. 
\end{proof}


\smallskip

%
%
%
%
%
%

%
%
%
%
%

\medskip






We conclude this section with a few trivial consequences of quantifier elimination in $T$. Some more interesting ones will be discussed in the next section (e.g., Lemma \ref{lem:complete_QE}).

First we note that some of the observations above become trivially true for any substructure of $\cC$ (which in our case just means a set containing all the individual constants), for example:

\begin{remark}\label{rem:QEA }Assume that $T$ has QE and let $A$ be a substructure of $\cC$.
\begin{enumerate}
\item
	Any externally definable subset of  $A$ is definable internally in $\cC|_A$. A $T$-type over $A$ is also a type 
	in $\cC|_A$. 
\item
	If $\cC|_A$ is $\lam$-compact, that is, it is a $\lam$-compact model of the theory $Th(\cal{C}|_A)$. Then  
	every $T$-type of size $<\lam$ which is finitely satisfiable in $A$ is realized in $A$. 	
\end{enumerate}
\end{remark}
 
 Quantifier elimination also adds to the understanding of $T^P$:
 
 \begin{remark}\label{rem:QEP} Assume that $T$ has QE and let $A$ be a substructure of $\cC$.
\begin{enumerate}
\item
	$T^P$ has QE.
\item
	Every subset of $P^\cC$ definable in $\cC$ (with or without parameters) is definable by the same quantifier
	free formula (with or without parameters, respectively) both in $\cC$ and in $\cC|_P$. 
\item
	Corollary \ref{co:P-type}(ii) can be strengthened to: both $\hat p$ and $p'$ are equivalent to the same quantifier 
	free type.

\end{enumerate}
\end{remark}

%
%

%

%
%

\bigskip 

%
%
%
%

From now on, to simplify the notation, when no confusion should arise, we will write $P$ for $P^\cC$. Also, for a set $A$, we will often denote by $A$ both the set and the substructure of $\cC$ with universe $A$. So for example, when we write that $A\cap P^\cC$ is $\lam$-saturated, or just that $A\cap P$ is $\lam$-saturated, we mean that the substructure $\cC|_{A\cap P^\cC}$ is a $\lam$-saturated model of the appropriate theory.

\section{Completeness and relevant types}

%

In trying to reconstruct $M$ from $M| P^M$ one needs to work with sets $A$ satisfying
$P^M\subseteq A\subseteq M$. Such $A$ have the following property (which under a certain assumption on saturation of $P^M$ characterizes such sets; see Proposition \ref{8} below), that can be viewed as an analogue to
Tarski-Vaught Criterion for being an elementary submodel:

\begin{de}\label{6}
$A\subseteq {\cal C}$ is {\it complete} if for every
formula $\psi(\bar x,\bar y)$ and $\bar b\subseteq A, \models
(\exists\bar x\in P)\psi(\bar x,\bar b)$ implies $(\exists
\bar a\subseteq P\cap A)\models \psi(\bar a,\bar b)$. 
\end{de}

The following useful characterization offers a different understanding of the notion of completeness:

\begin{obs}
\label{obs:complete_characterization}
A set $A$ is complete if and only if for every $\bar a\subseteq A$ and
$\phi(\bar x,\bar y)$  the $\phi$-type $tp_\phi (\bar a/P^{\cal C})$ is definable
over $A\cap P^{\cal C}$ and $A\cap P^{\cal C}\prec P^{\cal C}$.
\end{obs}

\begin{proof}
Assume that $A$ is complete. First we use the Tarski-Vaught criterion in order to show that $A\cap P^{\cal C}\prec P^{\cal C}$. Indeed, if $P^{\cal C} \models \exists x \theta(x, \b)$  with $\b \in A$, then (recall Remark \ref{rmk:externaltointernal}) $\cC\models \exists x P(x) \land  \hat\theta(x, \bar b)$. By completeness of $A$, there exists $c \in A$ such that $\cC \models P(c) \land  \hat\theta(c, \bar b)$, hence $P^{\cal C} \models \theta(c,\b)$, as required. 

Now let  $\bar a\subseteq A$. We know that $tp_\phi (\bar a/P^{\cal C})$ is definable
over $P^{\cal C}$; so suppose $\cC \models \forall \y \theta(\y, \c) \iff \phi(\a,\y)$ for some $\c \in P^{\cal C}$. By completeness, such a $\c$ exists already in $A\cap P^{\cal C}$. 

\smallskip

For the other direction, assume that $\cC \models (\exists \x \in P) \phi(\x, \a)$ with $\a \in A$. Since $tp_\phi (\bar a/P^{\cal C})$ is definable
over $A\cap P^{\cal C}$, the formula $\phi(\x, \a)$ is equivalent to a formula $\theta(\x, \d)$ for some $\d \in A\cap P^{\cal C}$. So $\cC \models (\exists \x \in P) \theta(\x, \a)$, hence $P^\cC \models \hat \theta(\c, \d)$ for some $\c$. Since $A\cap P^{\cal C}\prec P^{\cal C}$, such $\c$ exists also in $A\cap P^{\cal C}$; and we have $\models \theta(\c, \d)$, as required. 


\end{proof}

\smallskip

\begin{fact} \label{10}For any complete $A$ there are $\langle \Psi_\psi :\psi(\bar
x,\bar y)\in L(T)\rangle $ (depending on $A$) such that for all
$\bar a\subseteq A$, $tp_\psi (\bar a/P\cap A)$ is definable by
$\Psi _\psi (\bar y,\bar c)$ for some $\bar c\subseteq A\cap P$.
\end{fact}
\begin{proof}
By compactness, go to a $|T|^+$-saturated model. So, for each $\psi$ we
have but finitely many candidates $\Psi^1_\psi,\dots , \Psi ^n_\psi$ \red{(or else, by compactness, there is an undefinable type)}.
As without loss of generality $|P^{\cal C}|\geq 2$, we can manipulate
these as in \cite{Sh:c} II$\S$2  to an $\Psi _\psi$.
\end{proof}

\smallskip

The following properties of complete sets are clear
:

\begin{fact}\label{complete_simple}\label{7}
\begin{itemize}
\item[(i)] If $M\prec {\cal C}$ and $P^M\subseteq A\subseteq M$, then
$A$ is complete.
\item[(ii)] If $\langle B_i:i<\delta\rangle $ is an increasing sequence of complete sets, then
$\bigcup_{i<\delta} B_i$ is complete.
\end{itemize}
\end{fact}

Furthermore, if $T$ has QE, then the property of completeness for a set $A$  depends only on its first order theory (as a substructure of $\cC$):

\begin{lem}\label{lem:complete_QE}\label{7.5}(T has QE)
\begin{itemize}
\item[(i)] If $A_1\equiv A_2$, then $A_1$ is complete iff $A_2$ is
complete.
\item[(ii)] $A$ is complete iff whenever the
sentence $$\theta=:(\forall \bar y)[S(\bar y)\leftrightarrow (\exists x\in P)
R(x,\bar y)]$$ for quantifier free 
$R,S$ is satisfied in ${\cal C}$, then
$A$ satisfies $\theta$.

\medskip

Furthermore, if $T$ has QE down to the level of predicates (e.g., $T$ has been Morleyized), then it is enough to consider  all the formulas above with $R,S$ predicates. 
\end{itemize}
\end{lem}

\begin{proof}

\begin{itemize}

 \item[(i)] Follows easily from (ii), but we also give a direct proof. Assume that $A_1\equiv A_2$ and $A_1$ is complete. Let us prove this for $A_2$. 
 
 Let $\ph(\x,\y)$ be a formula. Since $T$ has QE, there are quantifier free formulae $\theta(\y)$ and $\theta'(\x,\y)$ such that:
 
 \begin{itemize}
 \item[(1)] \[ \cC \models \forall \y \left[ \left(\exists \x \in P  \ph(\x,\y) \right) \longleftrightarrow \theta(\y) \right] \]
 \item[(2)] \[ \cC \models \forall \x\y \left[  \ph(\x,\y)  \leftrightarrow \theta'(\x, \y) \right] \]
 \end{itemize}

Since $A_1$ is complete (combining the definition with (1) and (2) above), we have that whenever $\cC \models \theta(\a_1)$ for $\a_1 \in A_1$, there is some $\c \in A_1 \cap P$ such that $\cC \models \theta'(\c, \a_1)$. But since these formulae are quantifier free, clearly $A_1$ (as a substructure) satisfies 

\[ A_1 \models \forall \y \theta(\y) \longrightarrow \exists \x \in P \theta'(\x,\y) \]

Since $A_1 \equiv A_2$, so does $A_2$, which implies completeness. 

\item[(ii)] Similar [and easier]: If $A$ is complete and $\cC \models \theta$, assume that $A \models S(\b)$ for some $\b$; then (since $S$ is quantifier free), so does $\cC$, hence $\cC \models (\exists \x \in P) R(\x,\b)$, and by completeness of $A$, there is such $\x$ already in $P \cap A$. So $A \models \theta$. Conversely, assume that whenever $\cC \models \theta$, so does $A$. Let $\ph(\x,\y)$ be a formula, and assume that $\cC \models (\exists \x \in P) \ph(\x,\b)$.

As in the proof of (iv), let $R(\x,\y)$ and $S(\y)$ be quantifier free such that

\begin{itemize}
 \item[(1)] \[ \cC \models \forall \y \left[ \left(\exists \x \in P  \ph(\x,\y) \right) \longleftrightarrow  S(\y) \right] \]
 \item[(2)] \[ \cC \models \forall \x\y \left[  \ph(\x,\y)  \leftrightarrow R(\x, \y) \right] \]
 \end{itemize}
 
Clearly, $\cC \models \theta$ (with these specific $R$ and $S$), hence so that $A$. But  $\cC \models S(\b)$, hence  so does $A$; so $A \models (\exists \x \in P) R(\x,\b)$, hence $A \models R(\a, \b)$ for some $a \in A \cap P$, hence so does $\cC$, and we are clearly done.

\smallskip
The ``futhermore" part is trivial. 
\end{itemize}
\end{proof}

\smallskip

%

We make a few remarks on the relation between $P$ and the algebraic closure in complete sets. 

\begin{obs}\label{obs:acl}

 Let $A$ be a complete set. Then:
 \begin{enumerate}
 \item
 	$A \cap \acl(P^\cC) \subseteq \acl(P^A)$
\item
	$\acl(A)  \cap P^\cC = P^A$
\end{enumerate}
 \end{obs}
 \begin{proof}
 \begin{enumerate}
 \item 
  	Let $d \in A  \cap \acl(P^\cC)$ (in this proof we will distinguish between tuples and singletons). Then there is a formula $\ph(x,\a)$ with $\a \in P^\cC$ which has $k<\om$ solutions in $\cC$ such that $\models \ph(d,\a)$. Since $A$ is complete
	(see Observation \ref{obs:complete_characterization}), the type $\tp(d/P^\cC)$ is definable over $P^A$, so the set 
	\[
		\set{\y \in P \colon \ph(d,\y) \wedge \exists^k x \ph(x,\y)}
	\]
	
	is definable  in $\cC|_P$ by a formula $\theta(\y)$  over $P^A$. Again, since $A$ is complete, $P^A \elem P^\cC$, so $\theta(\y)$ 
	has a solution in $P^A$, which finishes the proof.

 \item
 	 	Let $d \in \acl(A)  \cap P$. So there is $\ph(x,\a)$ with $\a \in A$ with $k<\om$ solutions in $\cC$ such that $\ph(d,\a)$ 
	holds. Just like in the proof of clause (i), the set $D = \set{x \in P\colon \models \ph(x,\a)}$ is definable 
	in $\cC|_P$ over $P^A$ and is non-empty. Clearly, $D$ is finite. Since $P^A \elem P^\cC$, $D \subseteq P^A$, as 
	required. 

\end{enumerate}
 \end{proof}

\smallskip

The characterization of completeness (Observation \ref{obs:complete_characterization}), in combination with Corollaries \ref{co:P-type} and \ref{co:P-sat}, yields the following important property of complete sets:

\begin{co}\label{co:complete-types}
	Let $A$ be complete and let $p$ be a (partial) type over
	$A$ with $``\x \subseteq P" \in p$. Then $p$ is equivalent both to a $T$-type and a $T^P$-type over
	$P^A$.  If, in addition, $P^A$ is $\lam$-compact, and $|p|<\lam$, then $p$ is realized in $P^A$.
\end{co}

\smallskip

We are now ready to define one of the most basic objects of our study: the relevant notion of type for this context.

\begin{de}\label{6.5}\label{dfn:startypes}
\begin{itemize}
\item[(i)] For a complete set $A$, let
$$S_*(A)=\{tp(\bar c/A):P\cap (A\cup \bar c)=P\cap A {\rm\ and}\ A\cup \bar
c\ {\rm is\ complete}\}$$
\item[(ii)] $A$ is \emph{stable over $P$}, or simply \emph{stable}, if (it is complete and) for all $A^\prime$
with $A^\prime\equiv A$, $|S_*(A^\prime)|\leq |A^\prime|^{|T|}$.
\end{itemize}
\end{de}

\begin{remark}
\begin{enumerate}
\item Even though ``stability over $P$'' is a more appropriate and accurate name for our notion of stability of a set (and the term ``stable set'' exists in literature, and has a different meaning), since we  have only one notion of stability in this article (stability over $P$), we will mostly omit ``over P'' and simply write ``stable''. 
\item Sometimes we refer to types in  $S_*(A)$ as \emph{complete types over $A$ which are orthogonal to $P$}. 
\end{enumerate}

\end{remark}

\begin{remark} 
	Note that by Observation \ref{obs:acl}, if $A$ is complete and $\c$ a tuple such that $tp(\c/A) \in S_*(A)$, then $\acl(A \cup \c)  \cap P^\cC = P^A$.
\end{remark}

The last remark also follows from the following more general criterion for a complete type being a *-type:

\begin{obs}\label{8.5}
 Let $A$ be a complete set and $\c$ a tuple. Then $tp(\c/A) \in S_*(A)$ if and only if for every formula $\psi(\x,\a,\y)$ over $A$ we have 
 \[ \models \exists \x \in P\, \psi(\x,\a,\c) \Longrightarrow \exists \bar b \in P\cap  A \, \text{such that} \, \models \psi(\bar b,\a,\c) \]
 \end{obs}
 
 \begin{proof}
 	Let $tp(\c/A) \in S_*(A)$, and let $\psi(\x,\a,\y)$ be a formula as above satisfying $\models \exists \x \in P\, \psi(\x,\a,\c)$. Since the set $A \cup \c$ is complete, the type $\tp(\a\c/P^\cC)$ is definable over $P \cap (A \cup \c) = P \cap A$. Hence the set $D = \set{\x \in P \colon \models \ps(\x, \a, \c)}$ is definable in $\cC|_P$ over $P^A$. This set is not empty by the assumption. Since $A$ is complete, $P^A \elem P^\cC$, hence the set $P^A \cap D$ is nonempty as well, as required for the ``only if'' direction. 
	
	Now assume that the right hand of the equivalence holds. First we note that $P \cap (A \cup \c) = P \cup A$. Indeed, 
	if $d \in P \cup (A \cup \c)$, then the formula $x = d$ has a solution in $P$, hence in $P^A$, so $d \in P^A$. Now  obviously (since $A$ is complete) $P^{A \cup \c} \elem P^\cC$. By Observation \ref{obs:complete_characterization}, $A \cup \c$ is complete. 
 \end{proof}

\smallskip

%
%
%

\medskip

The following  lemma shows that, although the definition of types orthogonal to $P$ may seem quite strong, such types are not very hard to come by.

\begin{lem}(\underline{The Small Type Extension Lemma})\label{extension}\label{9}\label{le:typeextension}
If $A \prec \cal C$ is saturated (or just $A\cap P$ is $|A|$-compact) and $p(\bar x)$ is an $L(T)$-type over $A$
of cardinality $<|A|$,  then there is some $\red{p^*}(\bar x)\in S_*(A)$
extending $p$.
\end{lem}

\begin{proof} To prove this, we have to
show that $p$ is realized by some $\bar c\in {\cal C}$ such that $P^{\cal
C}\cap (A\cup \bar c)=P^{\cal C}\cap A$ and $A\cup \bar c$ is complete.
So we have to extend $p$ in such a way that whatever part of $p$ that can be
realized inside $P$, has to be realized in $P\cap A$. 

Let $\seq{\psi_i(\bar z, \bar b, \x)\colon i < |A|}$ list all the formulas over $A$. Now define 
inductively for $i<|A|$ and $\psi_i(\bar z,\bar b,\bar x)$
(so $\bar b\in A$), a consistent $T$-type $p_i(\bar x), 
|p_i|<|p|^++|i|^++\aleph _0$ and $p_i$
increasing continuously, making sure that the requirements are met.

Let $p_0=p$. If $p_i(\bar x)\cup \{ (\not\exists\bar z\in P)\psi (\bar
z,\bar b,\bar x)\}$ is consistent, let it be $p_{i+1}$. If not,
consider the
type $q(\bar z)=\{\exists \bar x\phi (\bar x)\wedge\psi(\bar z,\bar b,\bar x)\red{\colon \phi(\x) \in p_i }\}\cup\{\bar z\in P\}$.
By compactness and
saturation of $A$ there is some $\bar c\subseteq P\cap A$ such that
$p_{i+1}:=p_i(\bar x)\cup \{\psi (\bar c,\bar b,\bar x)\}$ is consistent. 

In fact, since $A$ is complete, by Corollary \ref{co:complete-types}, $q$ is equivalent to a $T^P$-type over $P \cap A$. Hence  (since $|q|<|A|$), it is enough to assume that $P\cap A$ is $|A|$-compact.

\smallskip

By compactness $\red{p^*} = \bigcup _ip_i$ is realized by some $\bar a\in {\cal C}$. By Observation \ref{8.5}, 
$\tp(\a/A) \in S_*(A)$, as required. 


\end{proof}

\medskip

This Lemma yields another characterization of completeness, justifying, in some sense, the original motivation behind this definition (see the discussion in the very beginning of this section). 

\begin{pr}\label{approx}\label{8}\label{prp:complete_exnetdabletomodel}
Suppose that $A\cap P$ is $|A|$-compact. Then $A$ is complete if and only if
there exists an $M\prec {\cal C}$ with $P^M\subseteq A\subseteq M$.
If $|A|=|A|^{<|A|}>|T|$, we can add ``$M$ saturated''.
\end{pr}
\begin{proof}
The direction $\Leftarrow$  is trivial.
For the other direction we
construct a model inductively, using Fact ~\ref{complete_simple} for the limit stages
and the previous lemma for the successor stages.
\end{proof}

%
%

%

\smallskip

As we want to reconstruct $M$ from $P^M$ using some complete $A$ as an
approximation, it makes sense to look at $S_*(A)$ as candidates for
types to be realized. So with $S_*(A)$ small ($A$ stable), this task appears to be easier, as there
is less choice. 

\smallskip

The rest of the paper is devoted to the study of stable (in particular complete) sets. 

\medskip

In light of Lemma \ref{7.5}, since, in the definition of a stable set, instead of looking at a specific set $A$, we need to consider the class of all subsets $A'$ of $\cC$ with $A' \equiv A$, it would be very useful to ensure that $T$ has quantifier elimination. Fortunately, by the discussion in the previous section (specifically, by Observation \ref{obs:QE_harmless}), we can Morleyize $T$ with not much cost. Therefore for the rest of the paper we strengthen Hypothesis 1 and add

\begin{hyp}\label{hyp:1.5}(\underline{\emph{Hypothesis 1'}})

\noindent 
\begin{itemize}
\item[(iii)]
	$T$ has quantifier elimination, even to the level of predicates.
\end{itemize}
\end{hyp}

%
%



\section{Stability and rank}

In this section we begin the investigation of stable sets, and introduce a notion of rank that ``captures'' stability. 

\smallskip 

%

\begin{de}\label{de:internallydef}
	We say that a type $p \in S(A)$ is \emph{internally definable} if for every formula $\phi(x,y)$, the set 
	$\set{a \in A\colon \ph(x,a) \in p}$ is internally definable in $\cC|_A$. We say that $p$ is internally definable over 
	$B \subseteq A$ if the set above is internally definable in $\cC|_A$ with parameters in $B$. 
\end{de}

\begin{de}\label{R}
For a complete set $A$, a (partial) $n$-type $p(\bar x)$ (with parameters in ${\cal C}$), 
  sets $\Delta _1,\Delta _2$ of formulas $\psi (\bar
x,\bar y)$, and a cardinal $\lambda$, we define when $R^n_A(p,\Delta
_1,\Delta _2,\lambda)\geq \alpha$. We usually omit $n$.

\begin{itemize}
\item[(i)] $R_A(p,\Delta _1,\Delta _2,\lambda)\geq 0$ if $p(\bar x)$ is consistent. 

\item[(ii)] For $\alpha$ a limit ordinal: $R_A(p,\Delta _1,\Delta
_2,\lambda)\geq \alpha$ if $R_A(p,\Delta _1,\Delta
_2,\lambda)\geq \beta$ for every $\beta <\alpha$.

\item[(iii)] For $\alpha =\beta +1$ and $\beta$ even:
For $\mu<\lambda$ and finite $q(\bar x)\subseteq p(\bar x)$ we can
find $r_i(\bar x)$ for $i\leq\mu$ such that;

\begin{itemize}
\item[{1.}] Each $r_i$ is a $\Delta _1$-type over $A$,

\item[{2.}] For $i\not= j, r_i$ and $r_j$ are explicitly contradictory
(i.e. for some $\psi$ and $\bar c$, $\psi (\bar x,\bar c)\in r_i,
\neg\psi(\bar x,\bar c)\in r_j$).

\item[{3.}] $R_A(q(\bar x)\cup r_i(\bar x), \Delta _1,\Delta
_2,\lambda)\geq \beta$ \red{for all $i$}.
\end{itemize}

\item[(iv)] For $\alpha =\beta +1: \beta$ odd: For 
$\mu<\lambda$ and finite $q(\bar x)\subseteq p(\bar x)$ and $\psi
_i\in \Delta _2, \bar d_i\in A$ ($i\leq \mu$), there are $\bar b_i\in
A\cap P$ such that $R(r_i,\Delta _1,\Delta
_2,\lambda)\geq \beta$ where $r_i=q(\bar x)\cup \{(\forall\bar
z\subseteq P) \left[\psi _i(\bar x,\bar d_i,\bar z)\equiv\Psi_{\psi
_i}(\bar z,\bar b_i)\right] \colon i<\mu\}$ where $\Psi_{\psi _i}$ is as in Fact \ref{10}.

\end{itemize}

$R^n_A(p,\Delta _1,\Delta _2,\lambda)= \alpha$ if $R^n_A(p,\Delta
_1,\Delta _2,\lambda)\geq \alpha$ but not $R^n_A(p,\Delta
_1,\Delta _2,\lambda)\geq \alpha +1$. $R^n_A(p,\Delta
_1,\Delta _2,\lambda)=\infty$ iff $R^n_A(p,\Delta
_1,\Delta _2,\lambda)\geq \alpha$ for all $\alpha$.
\end{de}

The main case for applications will be $\lambda =2$. Note that the
larger $R^n_A(p,\Delta _1,\Delta _2,\lambda)$, the more evidence there
is for the existence of many types $q(\bar x)\in S_*(A)$ consistent with
$p(\bar x)$.

\begin{fact}\label{rank}
\begin{itemize}
\item[(i)] The rank $R^n_A(p,\Delta _1,\Delta _2,\lambda)$ is
increasing in $A, \Delta _1$ and decreasing in $p,\Delta _2,\lambda$.

\item[(ii)] For every $p$ there is a finite $q\subseteq p$, such that 
$R_A(p,\Delta _1,\Delta _2,\lambda)=R_A(q,\Delta _1,\Delta _2,\lambda)$.

\item[(iii)] For any $A$ and any finite $\Delta _1,\Delta _2,\lambda, m$ and
$\psi (\bar x,\bar y)$ there is a formula $\theta (\bar y)\in L({\cal
C}| A)$ such that for $\bar b\subseteq A$ $R_A(\psi (\bar
x,\bar b),\Delta _1,\Delta _2,\lambda)\geq m$ iff $A\models\theta (\bar b)$. 
The formula $\theta$ doesn't really depend on $A$ but has quantifiers
ranging on $A$.

\end{itemize}
\end{fact}

\red{
\begin{fact}\label{rankeven}
	Let $A$ be complete, $p\in S_*(A)$, $q^* \subseteq p$, and assume
	 $$R^n_A(q^*,\Delta _1,\Delta _2,\lambda)=R^n_A(p,\Delta _1,\Delta _2,\lambda)=k < \infty$$ Then $k$ is even.
\end{fact}

\begin{proof}
	Assume $k$ is odd; we shall show that $R^n_A(q^*,\Delta _1,\Delta _2,\lambda)\ge k$ implies $R^n_A(q^*,\Delta _1,\Delta _2,\lambda)\ge k+1$. 
	
	 As in Definition \ref{R}(iv), let  
$\mu<\lambda$,  $q(\bar x)\subseteq q^*(\bar x)$ finite, $\psi
_i\in \Delta _2, \bar d_i\in A$ ($i\leq \mu$).

Let $\c \models p$; so 
$A \cup \set{\c}$ is complete by the assumption $p \in S_*(A)$. 

As $\bar d_i \in A$, clearly $A\cap P=(A\cup \bar c \bar d_i)\cap P$,
hence $tp(\bar c\bar d_i/A)\in S_*(A)$.
Hence $tp_{\psi_i}(\bar c\bar d_i/P^{\cal C})$ is defined by 
$\Psi _{\psi _i}(\bar y,\bar b_i)$ with $\bar b_i\subseteq P\cap(A\cup \bar
c\bar d_i)=A\cap P$, where $\Psi_{\psi _i}$ is as in Fact \ref{10}.

So $\theta_i(\bar x):=(\forall\bar y\subseteq P)
[\psi_i(\bar x,\bar d_i,\bar y)\equiv\Psi_{\psi_i}(\bar y,\bar b_i)]$
belongs to $tp(\bar c/A)=p$; hence 

$R_A(q\cup\{\theta_i(\bar
x)\},\Delta _1,\Delta _2,\lam)\geq R_A(p,\Delta _1,\Delta _2,\lam)=k$.

Now by (iv) of Definition \ref{R}, $R_A(p,\Delta _1,\Delta _2,\lam)\geq k+1$), and we are done. 

\end{proof}

}

\begin{theorem}\label{13}
The following are equivalent:
\begin{itemize}
\item[(i)] $A$ is stable.
\item[(ii)] For every finite $\Delta _1$ and finite $n$ there are some
finite $\Delta _2$ and finite $m$ such that 
$R^n_A(\bar x=\bar x,\Delta _1,\Delta _2,2)\leq m$.
\end{itemize}
\end{theorem}

\begin{proof}
$(ii)\Rightarrow (i)$: Suppose $(ii)$ holds. Since condition
(ii) speaks only about $Th({\cal C}|_A)$, it suffices to prove
$|S_*(A)|\leq |A|^{|T|}$ as the same proof works for every
$A^\prime\equiv A$. 

Let $\lambda :=|A|^{|T|}$ and assume that for $i<\lambda ^+$ there are
distinct types $p_i=tp(\bar c_i/A)\in S_*(A)$. By Fact \ref{rank} (ii)
for every $i<\lambda ^+$ and finite $\Delta _1,\Delta _2$ we can find
a finite $p=p_{i,\Delta _1,\Delta _2}\subseteq p_i$ such that
$R_A(p,\Delta _1,\Delta _2,2)=R_A(p_i,\Delta _1,\Delta _2,2)$.
Let $q_i=\bigcup _{\Delta _1,\Delta _2}p_{i,\Delta _1,\Delta _2}$.
So $q_i\subseteq p_i,|q_i|\leq |T|$ and by Fact \ref{rank} (i) 
$R_A(p_i,\Delta _1,\Delta _2,2)=R_A(q_i,\Delta _1,\Delta _2,2)$
for every finite $\Delta _1,\Delta _2$.

The function $F$ with $dom F=\lambda ^+, F(i)=\langle p_i|_\psi
:\psi\in L(T)\rangle $ is one-to-one where
$p|_\psi=\{\pm\psi(\bar x,\bar y):\pm\psi(\bar
x,\bar z)\in p\}$.
Hence, $\lambda ^+\leq \prod_{\psi\in L(T)} |\{p_i| \psi:
i<\lambda ^+\}|$.
If for every $\psi$, $|\{p _i|\psi :i>\lambda
^+\}|\leq\lambda$, we get a contradiction as $\lambda ^{|T|}=\lambda$.
Choose $\psi ^*$ such that $|\{p_i| \psi ^* :i<\lambda
^+\}|=\lambda ^+$. By renaming we can assume that $\{p_i|\psi
^*:i<\lambda ^+\}$ are pairwise distinct.

There cannot be more than $(|A|+|T|)^{|T|}=|A|^{|T|}=\lambda$ different
$q_i$, so without loss of generality, $q_i=q^*$ for all $i$. Also we can
assume that all $p_i$'s are $n$-types for some fixed $n$. Applying
condition (ii) to $n$ and $\Delta _1:=\{\psi ^*\}$ there is a finite set
$\Delta _2$ and $m<\omega$ such that $R^n_A(\bar x=\bar x,\Delta
_1,\Delta _2,2)\leq m$.

Let $k:=R^n_A(q^*,\Delta _1,\Delta _2,2)$. Since $q^*$ is consistent and
by monotonicity (Fact \ref{rank} (i)) and the fact that
$\{\bar x=\bar x\}\subseteq q^*$, it
follows that $0\leq k\leq m<\omega$.

\red{Recall also that by the construction of $q^*$,  for all $i$ \[ R^n_A(p_i,\Delta _1,\Delta _2,2) = R^n_A(q^*,\Delta _1,\Delta _2,2) = k \]

We are now going to show that $R^n_A(q^*,\Delta _1,\Delta _2,2) \ge k+1$, hence obtaining a contradiction. } Note that by \ref{rankeven}, $k$ is even. 

From now on, for simplicity of notation, let $R$ denote $R^n_A$. 


As $p_0|\psi ^*\not=p_1|\psi^*$, there is $\psi ^*(\bar
x,\bar b), \bar b\in A$ such that $\psi ^*(\bar x,\bar b)\in p_0$ and
$\neg\psi^*(\bar x,\bar b)\in p_1$ (or conversely). Now
$R(q^*\cup\{\pm\psi ^*(\bar x,\bar b)\},\Delta _1,\Delta _2,2)\geq
k$ as this type is contained in $p_0$ or $p_1$, hence by monotonicity
$R(q^*,\Delta _1,\Delta _2,2)=R(q^*\cup\{\pm\psi ^*(\bar x,\bar
b)\},\Delta _1,\Delta _2,2)$. So $R(q^*,\Delta _1,\Delta _2,2)>k$, a
contradiction.

This finishes the proof of one direction. 



%
%
%

In order to prove the other direction, assume that condition (ii) fails.
We will prove a
strong version of $\neg (i)$. Let $(ii)$ fail through
$\Delta _1$. So for all finite $\Delta _2$,
$R(\bar x=\bar x,\Delta _1,\Delta
_2,2)\geq \omega$.

Let $\lambda=\lambda^{<\lambda}>|T|$ (which exists by our set theoretic
assumption).
Let $B\equiv A$ be saturated, $|B|=\lambda$.

An $m$-type $p(\bar x)$ over $B$ (in $L(T)$) is called {\it large}
(for $\Delta _1$) if
for all finite $\Delta _2$, $R(p(\bar x),\Delta _1,\Delta _2,2)\geq
\omega$.

So $\neg (ii)$ says that $\bar x=\bar x$ is large for $\Delta _1$.
It will be enough to prove the following claims:

Assume $p(\bar x)$ is over $B$ and is large,
$|p|<\lambda$. Then the following holds:
\begin{itemize}
\item[(a)] For any $\bar b\in B,\psi =\psi (\bar x,\bar y,\bar z)$ there exists
$\bar d\subseteq P\cap B$ such that $p(\bar x)\cup\{\forall
\bar z\subseteq
P[\psi (\bar x,\bar b,\bar z)\equiv\Psi_\psi (\bar z,\bar d)]\}$ is
large.
\item[(b)] For $\bar b\subseteq B$
at least one of $p(\bar x)\cup\{\pm\psi(\bar x,\bar b)\}$ is large.
\item[(c)] For some $\psi(\bar x,\bar y)\in \Delta _1$ and $\bar b\in B$
we have $p(\bar x)\cup\{\pm\psi(\bar x,\bar b)\}$ are large.
\end{itemize}

Note that from $(a)-(c)$ \red{(and Fact \ref{7}(ii))} it follows that $|S_*(B)|=2^\lambda$; in fact,
even $|\{p|_{\Delta _1}:p\in S_*(B)\}|=2^\lambda$. 

\red{This already contradicts stability. But we can say more. } For at least one
$p\in S_*(B)$, $p|_{\Delta _1}$ is not definable. Hence (by [Sh8]) there
exists some $B$, $B\equiv A$, $|B|=\lambda$ such that $|S_*(B)|\geq
Ded(\lambda)$, assuming for simplicity that $Ded(\lambda )$ is obtained.
This is a strong negation of $(i)$.

It is left to show $(a)-(c)$:

$(a)$: Without loss of generality assume that $p(\bar x)$ is closed
under conjunction. So we have to find $\bar d$ such that for all
$\rho =\langle \Delta _2,n,\theta(\bar x,\bar e)\rangle$ where $n<\omega ,
\theta (\bar x,\bar e)\in p(\bar x)$, and $\Delta _2\subseteq L(T)$
finite,

$$ (*_\rho) \ \  R(\theta (\bar x,\bar
e))\wedge (\forall \bar z\subseteq P)(\psi (\bar x,\bar b,\bar
z)\equiv\Psi_\psi (\bar z,\bar d)),\Delta _1,\Delta _2,2)\geq n.$$

For every such $\rho$ there are $\bar e^*_\rho\subseteq B,\chi _\rho\in L(T$
such that for $\bar d\subseteq P\cap B, B\models \chi_\rho(\bar d,\bar
e^*_\rho)$ iff $*_\rho$ holds for $\bar d$.

As $B$ is $\lambda$-saturated, $|p(\bar x)|<\lambda$ it suffices to show
that for every relevant $\rho _1,\dots ,\rho _n$ there is $\bar
d\subseteq P\cap B$ satisfying  $*_{\rho _1}\wedge \dots \wedge *_{\rho
_n}$.

By monotonicity properties of rank and the fact that $p$ is closed under
conjunction, it is enough to consider one $\rho$. But $R(\theta (\bar
x,\bar e),\Delta _1,\Delta _2,2)=\omega >n+2$. So by the definition of
rank there is suitable $\bar d\subseteq P\cap B$ satisfying $*_\rho$.

$(b)$ follows from $(a)$ with $\bar z$ empty.

For $(c)$ assume first that $\Delta _1=\{\psi \}$. Repeat the proof of
$(a)$ conjuncting over $\pm\psi$, using the other clause in the
definition of rank.

If $|\Delta _1|>1$, assume $(c)$ doesn't hold. So for every
$\psi\in\Delta _1$ there is some finite $q_\psi\subseteq p(\bar x)$ such
that stops $p(\bar x)\cup\{\pm\psi (\bar x,\bar b\}$ from being large.
Now use $R(\bigcup_{\psi\in \Delta _1}q_\psi,\Delta _1,\Delta _2,2)\geq
n+2$ to get a contradiction.
\end{proof}

\red{

The following two corollaries follow from the proof of (the second direction of) Theorem \ref{13}.

\begin{co} 
	In Definition \ref{6}(iv), it is not necessary to consider all $A' \equiv A$. 
	More specifically, a complete set $A$ is stable if and only if $|S_*(A')| \le |A'|^{|T|}$ for some $A' \equiv A$ saturated, $|A'|>|T|$. 
\end{co}

\begin{co}
	Let $A$ be complete unstable and saturated of cardinality $\lam>|T|$. Then $|
	S_*(A)| = 2^{\lam}$. In fact, $|S_{*,\Delta}(A)| = 2^\lam$, where $\Delta$ is 
	some finite set of formulas, and 
	\[S_{*,\Delta}(A) = \set{p\rest \Delta\colon p \in S_*(A)}\] 
\end{co}
%

}

\medskip

From now on, we will often omit the superscript and the subscript in the rank $R^n_A$, and write simply $R$ (at least when 
$n$ and $A$ are easily deduced from the context).

\medskip

In conclusion, we observe that every type $p \in S_*(A)$ over a stable set $A$ is internally definable (see Definition \ref{de:internallydef}).


\begin{co}\label{14} 

\begin{enumerate}
\item
If $A$ is stable, then for every $\psi (\bar
x,\bar y)\in L(T)$ there is $\Psi _\psi$ in $L(A)$ such that if $p\in
S_*(A)$, then for some $\bar b\subseteq A, \Psi _\psi (\bar
y,\bar b)$ defines $p| \psi$ in $\cC_A$. 

Specifically, for every $\c \in A$, $\psi(\x,\c) \in p$ if and only if $A \models \Psi _\psi (\bar
c,\bar b)$.
 \item Moreover, if $|A|\geq 2$, then for every $\psi(\x,\y)$, there is a definition $\Psi_\psi(\x,\y)$ as above which works uniformly for all $B \equiv A$ and $p \in S_*(B)$. 
 \end{enumerate}
\end{co}
\begin{proof}
\begin{enumerate}
\item
Let $\Delta _1=\{\psi\}$. Then there is some finite $\Delta _2$
such that $R_A(\bar x=\bar x,\Delta _1,\Delta _2,2)=n^*<\omega$. Let
$\theta (\bar x)\in p(\bar x)$ such that
$n=R_A(\theta(\bar x),\Delta _1,\Delta _2,2)=R_A(p(\bar x),\Delta _1,\Delta
_2,2)\leq n^*$.

Recall that $n$ is even (Fact \ref{rankeven}). 
Since $\Delta _1=\{\psi\}$, there is no $\bar
b\in A$ such that $R(\theta (\bar x)\wedge\pm\psi(\bar x,\bar b),\Delta
_1,\Delta _2,2)=n$.

But $\psi\in p(\bar x)\Rightarrow R(\theta (\bar x)\wedge \psi,\Delta _1,\Delta _2,2)=n$.

So for $\bar b\subseteq A$, we have
$$\psi (\bar x,\bar b)\in p(\bar x)\Leftrightarrow R(\theta (\bar
x)\wedge \psi(\bar x,\bar b),\Delta _1,\Delta _2,2)\geq n .$$

By Fact \ref{rank}$(iii)$ the right hand side is definable in $A$. More precisely, the predicate 

\[
\Psi(\y) = \left[ R_A(\theta (\bar
x)\wedge \psi(\bar x,\bar y),\Delta _1,\Delta _2,2)\geq n\right]
\]
 is definable in $A$; note that it may have parameters in $A$.
 
This proves that for each type $p\in S_*(A)$, there is some $\Psi _\psi$.

\item 
Now use a compactness argument to find a uniform definition $\Psi _\psi$ for all $B\equiv A$ 
and $p\in S_*(B)$. The argument is pretty standard, but we have chosen to include it due to the (unusual) additional requirement of uniformity for all $B$.

Let $\psi(\x,\y) \in L(T)$. Expand $L(T)$ with a new monadic predicate $Q(y)$, which is interpreted  in $\cC$ as the set $A$. Let $T^*$ be the theory of the expanded monster model. Consider the following set of formulae in the expanded language:

\[
	T^* \bigwedge 
\left\{
\forall \z \in Q  \exists \y \in Q \; \psi(\x,\y) \not\leftrightarrow [Q \models \Psi(\y,\z)]\colon \Psi(\y,\z)  \in L(T) \right\}
\]

Assume that this set is finitely satisfiable. Let $M^*$ be a model of $T^*$ in which the tuple $\a$ realize the above set, and let $M\models T$ be reduct of $M^*$ to $L(T)$. Denote $B = Q^{M^*}$. We now have (all formulas and types below are in the original language $L(T)$):

\begin{itemize}
\item
	$B \equiv A$
\item
	For all $\Psi(\y,\z)
\in L(T)$ and $\c \in B$,  $\Psi(\y,\c)$ does not internally (in $B$) define $\tp_\psi(\a/B)$. 
	
\end{itemize}

So for every $\psi(\x,\y)$ there are finitely many $\Psi_i(\y,\z)$ such that for any $B \equiv A$ and any $\a \in \cC$, the type $\tp_\psi(\a/B)$ is defined by $\Psi_i(\x,\b)$ for some $i$ and $\b \in B$. Now (since $|A|\ge 2$, hence so is any $B \equiv A$), we can combine these $\Psi_i$'s into one formula $\Psi(\x,\y)$ that works for all $B \equiv A$ and $\a \in \cC$, as required. 
\end{enumerate}

\end{proof}



So we have obtained a uniform (for all $B \equiv A$) notion of definability of types in $S_*(A)$ for a stable set $A$. Note, however, that what we got is not the ``usual" notion of definability of types; this is different than saying that $p$ is definable (in $\cC$). 
Specifically, unless $A\prec {\cal C}$, $\Psi _\psi$ might have quantifiers. In order to obtain quantifier free definitions, we will need to make yet another ``structure'' assumption; see Hypothesis \ref{asm:2}. 
%
%
%
%



\section{Stability and primary models}

In this section our goal is to obtain a characterization of stable sets that strengthens the characterization of complete sets with a saturated $P$-part (Proposition \ref{prp:complete_exnetdabletomodel}). Specifically, we will show that if $A$ is (also) stable, then, in addition, one can have the model $M$ in Proposition \ref{prp:complete_exnetdabletomodel} be ``constructible'' over $A$ in a nice way. 

First, we strengthen the Small Type Extension Lemma (Lemma \ref{le:typeextension}) in this context. 

\begin{lem}\label{16} 
\begin{enumerate}
\item

Assume $B$ is stable,  $|B| = |P^B| = \lambda$, $P^B$ is saturated.
Let  $p(\bar x)$ 
an $m$-type over $B$, $|p(\bar x)|<\lambda$, then there is
$q(\bar x)$ such that $|q(\bar x)|\leq |T|$,
$p(\bar x)\cup q(\bar x)$ consistent and there is $r\in S_*(B)$ such
that $p(\bar x)\cup q(\bar x)\equiv r(\bar x)$. 

In particular,  $r(\bar x)$ is
$\lambda$-isolated.
\item
The previous clause is also true if $\x$ is an infinite tuple with $<\lam$ variables, but in this case we can only 
require that  $|q|<\lam$.

Specifically, if $|\x| = \ka < \lam$, then there exists $|q| \le |T|\cdot\ka$ as above.

\end{enumerate}

\end{lem}

\begin{proof} 

\begin{enumerate}
\item 
Let $\{\psi _i(\bar x,\bar y_i):i<|T|\}$ list all
formulas of $L(T)$. Let $\Delta _i$ be finite such that $R(\bar x=\bar
x,\{\psi _i\},\Delta _i,2)<\omega$ (where $R = R^m_B$). Define $q_i(\bar x)$ by induction
on $i < |T|$ such that
\begin{itemize}
\item[(a)] $q_i$ is finite and is over $B$,
\item[(b)] $p(\bar x)\cup \bigcup _{j\leq i}q_j(\bar x)$ is
consistent, and
\item[(c)] $R(p\cup\bigcup _{j\leq i}q_j,\{\psi _i\},\Delta _i,2)$ is
minimal with respect to (a) and (b).
\end{itemize}


By Lemma \ref{9}, there is $p^* \in S_*(B)$ extending $p\cup\bigcup _{j\leq i}q_j$. Clearly, $R(p^*,
\{\psi _i\},\Delta _i,2) \le R(p\cup\bigcup _{j\leq i}q_j,
\{\psi _i\},\Delta _i,2)$, and for some finite $q' \subseteq p^*$ we have $R(p^*,
\{\psi _i\},\Delta _i,2) = R(q',
\{\psi _i\},\Delta _i,2)$. If $R(p^*,
\{\psi _i\},\Delta _i,2) < R(p\cup\bigcup _{j\leq i}q_j,
\{\psi _i\},\Delta _i,2)$, setting $q'_i = q_i \cup q'$ would contradict the ``minimality'' of $q_i$ (clause (c) above). Hence $R(p^*,
\{\psi _i\},\Delta _i,2) = R(p\cup\bigcup _{j\leq i}q_j,
\{\psi _i\},\Delta _i,2)$. 


By Fact \ref{rankeven}, $R(p^*,
\{\psi _i\},\Delta _i,2)$ is even, hence so is $R(p\cup\bigcup _{j\leq i}q_j,\{\psi _i\},\Delta _i,2)$.
In particular, we have: 

\begin{itemize}
\item[(d)]
For no $\bar b\subseteq B$ do we have $R(p\cup\bigcup _{j\leq
i}q_j\cup \{\pm\psi_i(\bar x,\bar b)\},\{\psi _i\},\Delta
_i,2)\geq R(p\cup\bigcup _{j\leq i}q_j,\{\psi _i\},\Delta _i,2)$
\end{itemize}

Clearly $q = |\bigcup _{j\leq |T|}q_j|\leq |T|$. By 
Lemma \ref{extension} there is some $r\in S_*(B)$
such that $p\cup\bigcup_{j<|T|}q_j\subseteq r$. By (c) and (d) above it follows that $p\cup\bigcup_{j<|T|}q_j\vdash r$.

\item 
Let $\seq{\psi_i(\x_i, \y_i)\colon i<\ka}$ list all the formulas where $\x_i$ is a finite tuple from $\x$ (so $|T| \le \ka < \lam$), and define $q_i$ on induction on $\ka$ just as in the proof of the previous clause. Since each $q_i$ is finite, $q = |\bigcup _{j\leq |\ka|}q_j|\leq \ka<\lam$, and we can again use Lemma   \ref{extension} in order to obtain $r\in S_*(B)$ as required.
\end{enumerate}

\end{proof}

\begin{pr}\label{pr:isoext}
Assume that $B$ is stable,  $|B| = |P^B| = \lambda$, $P^B$ is saturated. Let $C \supset B$ such that:
\begin{itemize}
\item $C$ is complete
\item $P^C = P^B$
\item $|C\setminus B| < \lam$
\item $\tp(C/B)$ is $\lam$-isolated

\end{itemize}

Let  $p(\bar x)$ 
an $m$-type over $C$,  $|p(\bar x)|<\lambda$. Then there is a $\lam$-isolated
$r\in S_*(C)$ extending $p$.
\end{pr}
\begin{proof}
 	Let $\c = \seq{c_i\colon i<\ka}$ list $C\setminus B$ (so $\ka<\lam$), and let $B_0 \subseteq B$ be such that $|B_0| < \lam$ and $\tp(\c/B_0) \equiv \tp(\c/B)$. 
	
	Separating the parameters of $p$, we can think of $p(\x)$ as the type $p(\x,\c)$ over $B\c$. By replacing all 
	occurrences of  $c_i$ by a new variable $y_i$, we therefore obtain a type $p(\x,\y)$ over $B$. 
	
	Denote $\hat p(\x,\y) = p(\x,\y) \cup \tp(\c/B_0)$. By Lemma \ref{16}(ii), there is $\hat r \in S^*(B)$ extending $\hat p$ which is $\lam$-isolated. Let $\hat{\a} \hat \c$ realize $\hat r$ in $\cC$. Note that $\hat \c \equiv_B \c$. Let $\a \in \cC$ such that $\a \c \equiv_B \hat{\a}\hat \c$. Note that it is still the case that $\a C \cap P = P^B$ and $\a C$ is complete, so in particular $\tp(\a/C) \in S_*(C)$.  
	
	Finally, clearly $\tp(\a/C)$ is $\lam$-isolated: indeed, if $B_1 \subseteq B$ is such that $\tp(\a \c/B_1) \vdash \tp(\a \c/B)$, then $\tp(\a/\c B_1) \vdash \tp(\a/\c B ) = \tp(\a/C)$, so we are done.
\end{proof}

\begin{de}
Let $N$ be a model, $P^N \subseteq B \subseteq N$.
\begin{enumerate}
\item We say that a model $N$ is \emph{$\lam$-prime} over a $B$ if $N$ is $\lam$-saturated, and it can be elementarily 
embedded over $B$ into any $\lam$-saturated model containing $B$.  

\item 

	We say that $N$ is \emph{$\lam$-atomic} over $B$ if 
for every $\bar d\subseteq N$, $tp(\bar d,B)$ is
$\lambda$-isolated over some $B_{\bar d}\subseteq B, |B_{\bar
d}|<\lambda$.
\item
	We say that the sequence $\d = \{d_i:i<\alpha\} \subseteq N$ is a $\lam$-construction over $B$ in $N$ if for all $i<\al$, the type $tp\{d_i/B\cup\{d_j:j<i\})$ is $\lambda$-isolated.

\item
	We say that a set $C \subseteq N$ is $N$ is \emph{$\lam$-constructible} over $B$ in $N$ if 
	there is a $\lam$-construction $\d$ over $B$ in $N$. 
	
	In particular, we say that $N$ is
	\emph{$\lam$-constructible} over $B$ if 
there is a construction $N=B\cup\{d_i:i<\lambda\}$ such that for all $i<\lam$ the type 
$tp\{d_i/B\cup\{d_j:j<i\})$ is $\lambda$-isolated. 
\item 
	We say that a model $N$ is \emph{$\lambda$-primary} over $B$ if it is $\lam$-constructible and $\lam$-saturated. 
\end{enumerate}
\end{de}

\begin{remark}\label{rem:primary}
\begin{enumerate}
\item
 If $N$ is $\lam$-primary over $B$, then it is $\lam$-prime over $B$.
\item ($\lam$ regular)
 If $N$ is $\lam$-constructible over $B$ witnessed by a construction $N=B\cup\{d_i:i<\lambda\}$, then for every $\al<\lam$, $\tp(\{d_i:i<\al\}/B)$ is $\lam$-isolated. Hence $N$ is $\lam$-atomic over $B$.
\end{enumerate}
\end{remark}
%

Recall that given a complete set $B$ satisfying $|B| = \lam = \lam^{<\lam}$ with a saturated $P$-part, Proposition \ref{prp:complete_exnetdabletomodel} implies that $B$ can be extended to a saturated model of cardinality $\lam$ with the same $P$-part. We now show that in case $B$ is stable, this model can be chosen to be $\lam$-primary (hence $\lam$-prime) over $B$.

\begin{theorem}\label{th:primary}
	Assume that $B$ is a stable set, $|P^B| = |B| = \lam = \lam^{<\lam}$, $P^B$ is saturated. Then there is $N \supseteq B$ which is $\lam$-primary over $B$. 
\end{theorem}
\begin{proof}
By using Proposition \ref{pr:isoext} repeatedly, one constructs $N = B\cup\set{d_i\colon i<\lam}$ by induction on $i<\lam$ such that $\tp(d_i/B\cup\set{d_j\colon j<i})$ is 
$\lam$-isolated, while making sure that for all $i<\lam$, all types over subsets of $B_i = B\cup\set{d_j\colon j<i}$ of cardinality smaller that $\lam$ are realized by some $d_j$ for $j <i$.

Specifically, we construct by induction on $i$ a sequence $\lseq{\d}{i}{\lam}$ of sequences such that:

 \begin{itemize}
\item 
	$\d_i = \lseq{d}{\al}{\al_i}$ (where $d_\al$ is a singleton) with $\al_i < \lam$ (so in particular $|\d_i|<\lam$)
\item
	$\d_i$ is increasing and continuous with $i$ (so in particular $\al_i \le \al_j$ for $i<j)$
\item
	For every $i < \lam$, every $A \subseteq B$, $|A|<\lam$, every type over $A\cup \d_i$ is realized by some $d_\al$
\item
	For every $i < \lam$, the set $B_i = B \cup d_i$ is complete, and $P^{B_i} = P^B$
\item 
	For every $\al<\lam$, the type $\tp(d_\al/B \cup \set{d_\be\colon \be<\al})$ is $\lam$-isolated
\end{itemize}

If we succeed, then clearly $\d_\lam = \bigcup_{i<\lam} \d_i$ is a $\lam$-construction of a $\lam$-saturated model $N$ over $B$. 

\bigskip

Let $\d_0 = \seq{}$.

For $i$ limit, take unions. For $i = j+1$, let $B_j = B \cup \d_j$. Let the sequence $\seq{p_{j,\ga} \colon \ga \in [j,\lam)}$ list all the types over subsets of $B_j$ of cardinality $<\lam$ (recall that $\lam = \lam^{<\lam}$). Now consider the sequence 
$\seq{p_{\ell,j}\colon \ell<i}$. 

Recall that by induction, $\d_j = \lseq{d}{\al}{\al_j}$. Now for $\ell<i$ and $\al = \al_j + \ell$, let $d_\al$ realize $p_{\ell,j}$ such that $\tp(d_\al/B \cup \set{d_\be\colon \be<\al}) \in S_*(B \cup \set{d_\be\colon \be<\al})$ is $\lam$-isolated (this is possible by Proposition \ref{pr:isoext}).

Clearly, setting $\al_i = \al_j + i$, the sequence $\d_i = \lseq{d}{\al}{\al_i}$ is as required. 


\end{proof}

In the proof of Theorem \ref{th:primary}, the assumption that $\lam = \lam^{<\lam}$ was only used in in order to ensure that at any stage $i<\lam$ of the construction, the number of types over small subsets of $B_i$ is bounded by $\lam$. For a specific theory $T$, this assumption may hold for other cardinals $\lam$ (for instance, if $T$ has a saturated model of cardinality $\lam$).

So for example, the same proof as above gives the following stronger result:

\begin{co}\label{co:primary}
	Let $A$ be a stable set such that there exists a saturated model $N$ of cardinality $\lam$ containing $A$ with $P^N = P^A$. Then there exists $M$, $A \subseteq M \prec N$, $M$ is $\lam$-primary over $A$. 
\end{co}

\section{From Stability of Models: Quantifier Free Definitions}

The goal of this section is to establish the major technical tool of this paper: quantifier free definability of types orthogonal to $P$ over stable sets. However, as we have already pointed out at the end of section 5, for this we will need an additional hypothesis. 

In \cite{Sh234}, the first author has shown (see Theorems 2.10 and 2.12 there) that if there is an unstable \emph{model}, then there is a forcing extension in which there are many $M_i$ pairwise non-isomorphic with $M_i|P=M_j|P$ (all of cardinality $|P|>\aleph_0$). We will, therefore, following the Classification Theory guidelines, add yet another hypothesis to Hypothesis \ref{asm:1}. Specifically,  from now on we assume the following:

\begin{hyp}\label{asm:2} (\underline{Hypothesis 2}). 
	 Every $M\prec \cC$ is stable over $P$ (as in Definition \ref{dfn:startypes}(ii))
\end{hyp}
 
Now we are ready to prove that  *-types over all stable sets (not just models) are quantifier free internally definable, and therefore are also definable in $\cC$ in the usual sense. 

\smallskip

\begin{theorem}\label{4.1}\label{15} If $A$ is stable, $|A|\geq2$, then for every $\psi(\bar
x,\bar y)$ there is a quantifier free $\Psi _\psi (\bar y,\bar z)\in
L(T)$ such that whenever $B\equiv A$ and $ p\in S_*(B)$ then 
$p| \psi$ is
defined by $\Psi _\psi (\bar y,\bar d)$ for some $\bar d\subseteq B$,
i.e. $p| \psi=\{\psi(\bar x,\bar a):\bar a\in B, B\models
\Psi_\psi(\bar a,\bar d)\}$.
\end{theorem}

\begin{proof}
Let $\lambda=\lambda ^{<\lambda}, \lambda>|A|+|T|$ to make things simple.

Also note that if $A$ is stable and $\bar c$ is finite with
$tp(\bar c/A)\in S^n_*(A)$, then $A\cup \bar c$ is also stable
(as every $p(x)\in S^1_*(A\cup\bar c)$ gives rise to some type
$q(\bar x)\in S^{n+1}_*(A)$.)

Let $A,p$ be a counterexample, and $\bar c$ realize $p$.
We can find $B$ saturated of power $\lambda$ such that $({\cal
C}| (A\cup \bar c),A,\bar c)\prec({\cal C}|
(B\cup\bar c),B,\bar c)$. Clearly $B,\bar c$ form a counterexample
too, and in particular $tp(\bar c/B)\in S_*(B)$. We will arrive at a
contradiction by showing how to construct the required quantifier free
definition. By Proposition
\ref{approx} there is a model $M, P^M\subseteq B\cup \bar c\subseteq
M$. $Th(M,B,\bar c)$ has a
saturated model of power $\lambda$
preserving the relevant properties. So without loss of generality
$(M,B,\bar c)$ is saturated and $|M|=|B|=\lambda$.


By Theorem \ref{th:primary} and Remark \ref{rem:primary},
there is a $\lambda$-saturated model $N, P^N\subseteq B\subseteq
N$ such that for every $\bar d\subseteq N$, $tp(\bar d,B)$ is
$\lambda$-isolated, say over $B_{\bar d}\subseteq B, |B_{\bar
d}|<\lambda$, and a construction $N=\{d_i:i<\lambda\}$ such that
$tp\{d_i/B\cup\{d_j:j<i\})$ is $\lambda$-isolated. In particular, $N$ is $\lam$-prime over $B$.


Hence we can embed $N$ into $M$ over $B$. So without loss of
generality $N\prec M$, and in particular $P^M=P\cap B=P\cap N=P^N$ and
$tp(\bar c/N)\in S_*(N)$.

Hence there are formulas $\Psi _\psi(\bar x,\bar e_\psi)\in L(T),\bar
e_\psi\subseteq N$ defining $tp_\psi(\bar c/N)$ for $\psi\in L$.
Let $E=\bigcup_{\psi\in L(T)}\bar e_\psi\subseteq N$ and $B^*=\bigcup
_{\bar d\subseteq E}B_{\bar d}$. So $|E|\leq |T|$ and $|B^*|<\lambda$.

Now, if  $\bar b_1,\bar b_2\in B$ realize the same type over $B^*$ (in
${\cal C}$), then they realize the same type over $B^*\cup E$ by choice
of $E$. Hence, they  realize the same type over $B^*\cup E\cup \bar c$.

For $\psi\in L(M,B)$ let

$$\Gamma _\psi =\{\psi(\bar c,\bar y_1)\equiv\neg\psi(\bar c,\bar
y_2)\}\cup \{\chi(\bar y_1,\bar d)\equiv\chi (\bar y_2,\bar d):\chi\in
L(T),\bar d\subseteq B^*\}\cup \{\bar y_1\bar y_2\subseteq B\}$$.

By the previous observation $\Gamma _\psi$ is not \red{realized in $(M,B)$}. \red{By the fact that
$(M,B)$ is $\lambda$-saturated, and $|\Gamma _\psi|<\lambda$, it is inconsistent}. By compactness there are
$\chi _1,\dots ,\chi _n\in L(T)$ and $\bar d_1,\dots ,\bar d_n\in B^*$
such that

$$ \Gamma^1_\psi=\{\psi(\bar c,\bar y_1)\equiv \neg\psi (\bar c,\bar
y_2)\}\\
\cup\{\bar y_1\bar y_2\subseteq B\}\cup\{ \chi _l(\bar y_1,\bar
d_l)\equiv\chi _l (\bar y_2,\bar d_l):l=1,\dots ,n\}$$

is inconsistent.

So we can define $tp _\psi (\bar c/B)$ since

$$\models \psi (\bar c,\bar b)\Leftrightarrow [\{l:1\leq l\leq n,
\models\chi _l(\bar b,\bar d_l)\} {\rm \  is \ in} P^*] $$

for some appropriate $P^*\subseteq {\cal P}\{1,\dots ,n\}$.
Now apply compactness as in [Sh:c,II\S 2].
\end{proof}

Note that we have used the assumption that models are stable in the
proof.

\medskip

\begin{theorem}\label{17}\label{stable} Let $A$ be complete and $\lambda =\lambda^{<\lambda}$.
The following are equivalent:
\begin{itemize}
\item[(i)]  $A$ is stable.
\item[(ii$)_\lambda$ ] If $A^\prime\equiv A$ is $\lambda$-saturated,
$\lambda =|A^\prime|>|T|$, then over $A^\prime$ there is a
$\lambda$-primary 
model
$M$.
\item[(iii$)_\lambda$] If $A^\prime\equiv A$ is $\lambda$-saturated,
$\lambda >|T|$, then every $m$-type $p$ over $A$, $|p|<\lambda $ can be
extended to a $\lambda$-isolated $q\in S_*(A^\prime)$.
\item[(iv)] For every $A^\prime\equiv A$ and $p\in S_*(A)$ and $\phi\in
L(T)$, $p| \phi$ is definable by some $\Psi_\phi (\bar y,\bar
a),\bar a\subseteq A,\Psi _\phi\in L(T)$.
\item[(v)] There is some collection $\langle \Psi _\phi;\phi\in L\rangle $ such that
for every $A^\prime\equiv A,p\in S_*(A^\prime)$ and $\psi\in L(T),
p|\psi$ is definable by $\Psi_\psi (\bar y,\bar a)$ for some
$\bar a\in A^\prime$.
\end{itemize}
\end{theorem}

\begin{rm} So $(ii)_\lambda, (iii)_\lambda$ do not depend on
$\lambda$. 
\end{rm}

\begin{proof}
Included in the proofs of Theorem \ref{15}, Lemma \ref{16}, and Theorem \ref{th:primary}.
\end{proof}

\begin{theorem} ($T$ countable) If $A$ is stable, $\bar a\in A$, and
$\models\exists x\theta (\bar x,\bar a)$, then there is $p\in S_*(A)$
such that $\theta (\bar x,\bar a)\in p$ and for every $\phi\in L(T)$
there is $\psi(\bar x,\bar a^\prime)\in p$ such that $\psi (\bar x,\bar
a^\prime)\vdash p|\phi$ (i.e. $p$ is locally isolated, i.e
${\bf F}^l_{\aleph _0}$-isolated. So the locally isolated types are
dense in $S_*(A)$.)
\end{theorem}
\begin{proof}
Again this is contained in the proofs of Theorem \ref{15} and Lemma \ref{16}.
\end{proof}

\section{Stationarization and Independence}

The following definition mimics the Tarski-Vaught criterion, when one does not demand
$A,B\prec {\cal C}$.

\begin{de}
$A\subseteq _t B$ if for every $\bar a\in A,\bar b\in B$ 
and $\psi\in L(T)$ such that $\models\psi(\bar b,\bar a)$ there is  
some $\bar
b^\prime\subseteq A$ such that $\models \psi(\bar b^\prime,\red{\bar a})$
\end{de}

\noindent
 As a simple example, note that $A$ is complete if and only if $A\cap P\subseteq _t P$.

\medskip

We can now define ``free'' (``non-forking'') extensions for *-types over stable sets. Such extensions will  be 
defined only to supersets that are ``elementary extensions'' in the sense defined above. 

The use of the term ``non-forking'' above is not just by analogy with classical stability theory, but (at least under certain circumstances, e.g., when $A$ is a model) has a precise technical meaning, as the notion of independence defined below coincides with the usual non-forking independence; see Corollary \ref{co:nonforking}.

\begin{de}

Suppose $A$ is stable, $p\in S_*(A)$ and $A\subseteq _t B$.
Then $q\in S(B)$ is a {\it stationarization} of $p$ over $B$ if 
for every $\psi\in L$ there is some definition $\Psi _\psi(\bar y,\bar a
_\psi)$ with $\bar a_\psi\subseteq A$ that defines both $p_\psi$ and $q _\psi$.
\end{de}

\begin{no}\label{no:ind}
\begin{enumerate}
\item
	We write $\a \ind_A B$ if $A$ is stable and $q=\tp(\a/B)$ is a stationarization of $p = \tp(a/A)$ (so in particular $p \in S_*(A)$ and $A \subseteq_t B$). In this case, will also write $q = p|B$.
\item
	We write $C \ind_A B$ if for every $\a \in C$  we have $\a \ind_A B$.
\end{enumerate}
\end{no}

\smallskip

\noindent
If $\a \ind_A B$ or $C \ind_A B$, we say that $\a$ (or $C$) is \emph{independent} from $B$ over $A$.

\bigskip

Let us point out some basic properties of the notions defined above. 

\red{
\begin{lem}\label{saturated-t}
\begin{enumerate}
\item 
	$A \subseteq_t B$ if and only for every quantifier free formula $\ph(\x)$ over $A$, if there exists $\b\in B$ such that $\models \ph(\b)$, then $A \models \exists \x \ph(\x)$.
\item
	If $A$ is $\lam$-saturated then $A \subseteq_t B$ if and only for every (partial) type $p(\x)$ over a subset of $A$ of size $<\lam$, if $p$ is realised by some  $\b \in B$, then it is realised by some  $\a \in A$.
\end{enumerate}
\begin{proof}
\begin{enumerate}
\item 
	By quantifier elimination and the assumption that there are no function symbols (so every subset is a substructure). 
\item
	By quantifier elimination, $p(x)$ is equivalent (in $\cC$) to a quantifier free type $\Delta(\x)$. By part (i), $\Delta$ is finitely satisfiable in $A$, and by saturation there exists $\a \in A$ such that $A\models \theta(\a)$ for all $\theta(\x) \in \Delta$. Since $\Delta$ is quantifier free, the truth value of $\theta(\a)$ is preserved between $A$ and $\cC$; so $\Delta(\x)$ is realised by $\a$, hence so is $p(\x)$.
\end{enumerate}

\end{proof}

\end{lem}
}

\begin{lem}\label{station}\label{20} Assume $A$ is stable, $A\subseteq _t B$ and
$p\in S_*(A)$. Then:
\begin{itemize}
\item[(i)] $p$ has a stationarization $q$ \red{over $B$}.
\item[(ii)] It is unique: We can replace "some $\Psi _\psi(\bar y,\bar a
_\psi)$" by "every...", so $q$ does not depend on its choice. 
\item[(iii)] If $B$ is complete, $q\in S_*(B)$.
\end{itemize}
\end{lem}

\begin{proof}
(i) By Theorem \ref{4.1} there are quantifier free formulas $\Psi 
_\psi(\bar y,\bar a_\psi)$ with $\bar a_\psi\in A$ defining $p|_\psi$.
Let $q=\{\psi(\bar x,\bar b);\bar b\subseteq B$ and $\models \Psi
_\psi(\bar b,\bar a _\psi)\}$.

$q$ is consistent: If not, there are $n<\omega,\psi _l(\bar x,\bar b
_l)\in q (l=1,\dots ,n)$ such that $\models \neg \exists\bar x(\bigwedge
_{l=1}^n\psi _l(\bar x,\bar b_l))$. So $\models\bigwedge _{l=1}^n\
\Psi _{\psi _l}(\bar b_l,\bar a_{\psi _l})
\wedge \neg\exists\bar x(\bigwedge _{l=1}^n\psi _l(\bar x,\bar b_l))$.
This is a formula in $L(T)$.
So there are $\bar b_l\in A, (l=1,\dots ,n)$ such that $\models\bigwedge
_{l=1}^n\Psi _{\psi _l}(\bar b^\prime _l,\bar a_{\psi
_l})\wedge\neg\exists \bar x(\bigwedge _{l=1}^n\psi _l(\bar x,\bar
b^\prime_l))$.         But $\bigwedge _l \psi _l(\bar x,\bar b^\prime
_l)\in p$, contradicting the consistency of $p$.

$q$ is complete: Assume  $\bar b\subseteq B,\psi\in L$ and $\psi (\bar
x,\bar b), \neg\psi(\bar x,\bar b)\notin q$. So $\models\neg\Psi _\psi(\bar
b,\bar a _\psi)\wedge\neg\Psi _{\neg\psi}(\bar b,\bar a _{\neg\psi})$.
By the definition of $q$ and $A\subseteq _t B$ there is  some $\bar
b^\prime\subseteq A$ so that $\models \neg\Psi_\psi (\bar b^\prime,\bar
a_\psi)\wedge\neg\Psi_{\neg\psi}(\bar b^\prime,\bar a _{\neg\psi})$. But
then $\psi (\bar x,\bar b^\prime),\neg\psi (\bar x,\bar b^\prime)\notin
p$ and $p$ is complete, a contradiction.

(ii) Same proof: Let $\Psi, q$ be as in (i), and suppose that $\Psi'
_\psi(\bar y,\bar a'_\psi)$ (not necessarily quantifier free) with $\bar a'_\psi\in A$ also defines $p|_\psi$, and assume that it defines a $\psi$-type $q'_\psi$ over B. If $q' \neq q$, that is, for example, $\psi(\x,\b) \in q$, $\psi(\x,\b) \not\in q'$, then $B \models \Psi_\psi(\b) \land \neg\Psi_\psi'(\b)$. Since $A \subseteq_t B$, there exists $\a \in A$ such that $\Psi_\psi(\a) \land \neg\Psi'_\psi(\a)$, which is clearly absurd, since both schemata $\Psi$ and $\Psi'$ define the same type $p \in S(A)$.

(iii) 
Let $B$ be complete. We
consider $\phi (\bar x,\bar y,\bar z)$ with $\phi ^\prime (\bar x,\bar
b):=(\exists\bar z\in P)\phi (\bar x,\bar b,\bar z)\in q, \bar b\in
^{\omega >} B$. We have to show that for some $\bar c\in ^{\omega >}
(B\cap P),\phi (\bar x,\bar b\bar c)\in q$. Without loss of generality
$\phi (\bar x,\bar y,\bar z)\vdash \bar z\subseteq P$. So in the
problematic case $\bar b\in ^{\omega >} B,\models \Psi _{\phi ^\prime}
(\bar b)$, but for no $\bar c\subseteq \red{{}^{\omega >}} (B\cap P)$ is
$\models\Psi_\phi (\bar b,\bar c,\bar a_\phi)$. But since $B$ is complete
this implies $\models \Psi _{\phi ^\prime}(\bar b)\wedge (\neg\exists\bar
z\subseteq P)\Psi _\phi ^\prime (\bar b,\bar z)$, so this is satisfied
by some $\bar b^\prime\in A$ \red{(since the definitions are over $A$ and $A\subseteq_t B$)}, and we get a contradiction \red{to $p \in S_*(A)$}.
\end{proof}
\begin{co} (of the proof):
If $A\subseteq _t B, A$ stable, $\bar c\subseteq\bar b,tp(\bar b/A)\in
S_*(A)$, then $tp(\bar c/A)\in S_*(A)$ and the stationarization of
$tp(\bar b/A)$ over $B$ includes the stationarization of $tp(\bar c/A)$
over $B$.
\end{co}

\red{
\begin{co} 
	Let $q \in S_*(B)$  definable over $A \subseteq_t B$, $A$ a stable set. Then $q$ is the stationarization of $q\rest A$. 
\end{co} 
}

The following example illustrates the importance of the condition $A \subseteq_t B$ in the definition of the stationarization (and in Lemma \ref{station}(i)). 

\begin{example}
	In \cite{Chatzidakis2020RemarksAT} Chatzidakis explores the theory $T=ACFA_0$ over $P = Fix(\sigma)$ and proves that for all uncountable $\lam$, if $A$ is a substructure with $P^M$ $\lam$-saturated, there exists a $\lam$-primary model $N$ of $T$ over $A$. This result immediately implies (by e.g. Theorem \ref{stable}) that $ACFA_0$ is stable over $P$. In this example, in the construction of the primary model over $A$, one has to address the following situation: $A \subseteq B$, $B^P = M^P$, $B$ is complete and stable (this is not specifically stated in  \cite{Chatzidakis2020RemarksAT}, but is a posteriori clear, since the construction in particular yields a $\lam$-primary model $N$ over $B$), $d \in B$, and the type (extending) the difference equation $\sigma(x) = d\cdot x$ is not realized in $B$. Clearly (by e.g. Lemma \ref{le:typeextension}), this types extends to $p \in S_*(B)$, and one has to realize $p$ in order to complete the construction of $N$. 
	
	Now consider $a \models p$. One may ask: can we have $b \in N$ such that $b\models p|Ba$ (so $a, b \models p$ such that $a\ind_B b$)? The answer is clearly no, since in this case $\frac{a}{b}$ would be a new element in $P$. So why does $N$ not realize the stationarization of $p$ to $Ba$? The issue is that if $B$ does not realize $p$, then $B \not\subseteq_t Ba$. And indeed in this case such a stationarization does not exist. 
%
\end{example}

\red{
Let us make a few further remarks. 

\begin{remark} 
\begin{enumerate}
\item
	If $A \subseteq B \subseteq C$ and $A \subseteq_t C$, then $A\subseteq_t B$. In particular, if $M$ is a model, then $M \subseteq_t C$ for any $C \supseteq M$. So $p \in S_*(M)$ has a (unique) stationarization over any superset. 
\item
	If, under the assumptions of \emph{(i)}, $q$ is the stationarization of $p \in S_*(A)$ over $C$ (so in particular $A$ is stable), and $B$ is stable, then $q\rest B$ is the stationarization 
	of $p$ over $B$.
\item
	If $A \subseteq_t B$, $tp(ab/B)$ is the stationarization of $tp(ab/A) \in S_*(A)$, and $Ab$ is stable, then $tp(a/Bb)$ is the stationarization of $tp(a/Ab)$
\item
	In the previous clause, if $b$ is finite, then the assumption on $Ab$ is redundant, that is, it follows from the other assumptions. 
\end{enumerate}
\end{remark}
\begin{proof} Easy.  For clause (iv), note that since $tp(ab/A) \in S_*(A)$, the set $Ab$ is complete; and since $A$ is stable and $b$ is finite,  $Ab$ is stable as well. 
%
\end{proof}
}

\red{
\begin{lem}\label{mapextension}
	Let $A,B,C$ be sets such that $A\subseteq_t B$, $A\subseteq C$, 
	$C\ind_A B$ (see Notation \ref{no:ind}). Let $F$ be an elementary map from $B$ onto $B'$, $G$ be an elementary map from $C$ onto $C'$ such that $F\rest A = G\rest A$. Then $F \cup G$ is elementary.
\end{lem}
\begin{proof}
	Let $A' = F(A) = G(A)$. Let $C''$ be such that $tp(C''/B)$ is the stationarization of $tp(C'/A)$. Then clearly there is an elementary map $G'$ such that $G'(C') = C''$, and $G'' \rest A' $ is the identity.  First we claim that $F \cup (G' \circ G)$ is elementary. 
	
	Indeed, let $\c \in C$. Then for every $\ph(\x,\y)$ there exists $\d \in A$ such that $\Psi_\ph(\y,\d)$ defines the $\ph$-type of $\c$ over $A$. In other words, $\ph(\c,\a)$ if and only if $\Psi_\ph(\a,\d)$ for all $\a \in A$. Then (since $G$ is elementary) we have $\ph(\c',\a')$ if and only if $\Psi_\ph(\a',\d')$ for all $\a' \in A'$, where $\c'\d' = G(\c\d)$. Hence (since $G'$ is elementary) we have $\ph(\c'',\a')$ if and only if $\Psi_\ph(\a',\d')$ for all $\a' \in A'$, where $\c'' = G'(c)$. Recalling that $G'(\a'\d) = \a'\d$, we get: 
	
	\[ \ph(\c'',\a'') \iff \Psi_\ph(\a'',\d'') \] for all $\a'' \in A''$, where $\c''\d'' = G'\circ G(\c\d)$.
	
	In other words, $tp(\c''/A')$ is definable. Since $tp(c''/B')$ is the stationarization of $tp(\c''/A')$, by Lemma \ref{20}, the same definition works for since $tp(c''/B')$. Hence for all $\b'' \in B'$, letting (as before) $\c''\d'' = G'\circ G(\c\d)$
	
	\[ \ph(\c'',\b'') \iff \Psi_\ph(\b'',\d'') \]
	
	In particular, the equivalence above holds for $\b'' = F(\b)$. Recall that since $\d \in A$, we have $F(\d) = G(d) = G'\circ G(\d)$. Hence \[ \ph(\c'',\b'') \iff \Psi_\ph(\b'',\d'') \iff \Psi_\ph(F(\b),F(\d)) \]
	(the rightmost equivalence holds since $F$ is elementary). Now, by the choice of $\Psi_\ph$, we also have 
	\[ \ph(G'\circ G(c),F(b)) = \ph(\c'',\b'') \iff \Psi_\ph(F(\b),F(\d)) \iff \ph(b,d) \]
	
	which proves that $F \cup (G' \circ G)$ is elementary. 
	
	Now clearly \[F\cup G = (G')^{-1}\circ\left[F \cup (G' \circ G)\right] \] is also elementary, and we are done.
	
\end{proof}
}

\red{
\begin{lem}\label{isolatedimplies} 
 Let $A$ be $\lam$-saturated and stable, $A \subseteq_t B$, $N$ a $\lam$-saturated model $\lam$-atomic over $A$ such that 
$N \ind_A B$.
 
 Then $tp(N/A) \vdash tp(N/B)$; so the types $tp(N/A)$ and $tp(B/A)$ are weakly orthogonal.  
\end{lem}
\begin{proof}
	Let $\c \in N$ and $\c' \in \cC$ such that $tp(\c/A) = tp(\c'/A)$. Assume towards contradiction that for some formula $\ph(\z,\b)$ (with $\b \subseteq B$) so that $\models\ph(\c,\b)\land\neg\ph(\c',\b)$. 
	
	The type $tp(\c/A)$ is isolated by $\Theta(\z)$, a partial type over a subset of $A$ of cardinality less than $\lam$. 
	
	Consider the following partial type:
	
	\[ \pi(\y)  = \left\{ \exists \z\z' \left[\theta(\z)\land\theta(\z') \land \ph(\z,\y)\land\neg\ph(\z',\y) \right]\colon\; \theta \in \Theta \right\} \]
	
	It is realized by $\b \in B$, hence, by \ref{saturated-t}(ii), it is also realized by some $\a \in A$. 
	
	Now consider the following type:
	
	\[ \left\{ \theta(\z)\land\theta(\z') \land \ph(\z,\a)\land\neg\ph(\z',\a) \colon \; \theta \in \Theta \right\} \]
	
	It is finitely satisfiable in $N$ (precisely because $\a \models \pi(\y)$), and since $N$ is $\lam$-saturated, it is realized by some $\c_1,\c_2$ in $N$ (recall that $\Theta$ is over a ``small'' set). But $\Theta(\z)$ implies a complete type over $A$; a contradiction.

\end{proof}
}



\section{Main consequences}

\begin{theorem} \label{22}\underline{(Stable Amalgamation for models).} If $M_l, l=1,2$ is saturated
of power $\lambda$ 
(or just $P^{M_0}$ is saturated), $P^{M_l}\subseteq M_0\prec M_l$, then we can find
$M\supseteq M_0\supseteq P^M$ and elementary embeddings $f_l$ of
$M_l$ into $M$ over $M_0$ such that $tp(\bar c/f_2(M_2))\in S_*(f_2(M_2))$ for all
$\bar c\in f_1(M_1)$, and, moreover it is the stationarization of $tp(\bar
c/M_0)$ over $f_2(M_2)$; that is, $f(M_1) \ind_{M_0} f(M_2)$. 

\red{If $\lam = \lam^{<\lam}$, then $M$ can be chosen to be saturated.}
\end{theorem}
\begin{proof}
We can find an elementary mapping $f_1$ from $M_1$ to ${\cal C}$ such
that $f_1|_{M_0}=id$ and  for all $\bar c\subseteq M_1$,
$tp(f_1(\bar c)/M_2)$ is the stationarization of $tp(\bar c/M_0)$:
\red{Since for $\bar c\in M_1, P^{M_1}\subseteq M_0\cup\bar c\subseteq M_1$, $M\cup \c$ is complete, hence 
$tp(\bar c/M_0)$ is in $S_*(M_0)$, and by Lemma \ref{20} has a stationarization $q_{\bar
c}$ over $M_2$ ($M_0 \subseteq_t M_2$, of course)}. By the previous corollary all these types $q_{\bar c}$
are compatible (being \red{a directed system}), so we can define $f_1$ as an elementary
map so that $f_1|_{M_0}=id, dom f_1=M_1$ and $f_1(\bar c)$
realizes $q_{\bar c}$: so for $c_1,\dots ,c_n\in M_1$, $M_2\cup
\{f_1(c_1),\dots ,f_1(c_n)\}$ is complete, $P\cap
(M_2\cup\{f_1(c_1),\dots ,f_1(c_n)\})=P\cap M_2$ by Lemma
\ref{station} (iii). Hence by Fact \ref{7}(i) $M_2\cup f_1(M_1)$ is
complete, $P\cap (M_2\cup \{f_1(c):c\in M_1\})=P\cap M_2=P\cap M_0$.
As $P\cap (M_2\cup f_1(M_1))$ is saturated of power $\lambda$, by
Proposition \ref{approx} there is some $M$ with $M_2\cup f_1(M_1)\subseteq M$ and
$P^M=P^M\cap (M_2\cup f_1(M_1))=P^{M_0}$ as required.

If $\lam = \lam^{<\lam}$, then in Proposition \ref{8}, $M$ can be chosen to be saturated. 
\end{proof}

We can now deduce amalgamation over stable sets with a saturated $P$-part. 

\begin{theorem}\label{23}\label{amalgamation}\underline{(Stable Amalgamation for types over stable sets).} Let $A$ be stable and saturated (\red{or just $A\cap P$ is $|A|$-compact}).
If $|T|<\lambda^{<\lambda}=\lambda=|A|$, then we have amalgamation in
$S_*(A)$. That is, if $tp(\bar a\bar b/A)\in S_*(A), tp(\bar
a\bar c/A)\in S_*(A), \bar a,\bar b,\bar c$ of length $<\lambda$, then
for some $\bar a^\prime, \bar b^\prime, \bar c^\prime, tp(\bar
a^\prime\bar b^\prime/A)=tp(\bar a\bar b/A), tp(\bar a^\prime\bar
c^\prime/A)=tp(\bar a\bar c/A)$ and $tp(\bar a^\prime\bar b^\prime\bar
c^\prime/A)\in S_*(A)$.
\end{theorem}
 \begin{proof} \red{Note that $P\cap A\a\b = P\cap A$ is saturated, $A\a\b$ is complete. Hence by Proposition \ref{8}} we can find a model $M_{\bar b},\lambda$-saturated of
cardinality $\lambda$ such that $A\cup\bar a\bar b\subseteq M_{\bar b}$,
and $P^{M_{\bar b}}\subseteq A$. Similarly, we can choose $M_{\bar
c},\lambda$-saturated of cardinality $\lambda$ with $A\cup\bar
a\bar c\subseteq M_{\bar c}$, and $P^{M_{\bar c}}\subseteq A$. By Theorem
\ref{17} (ii) there is a model $M_{\bar a}$ of cardinality 
$\lambda$, $A\cup\bar a\subseteq M_{\red{\bar a}}$,
and $P^{M_{\bar a}}\subseteq A$ such that $M_{\bar a}$ is $\lambda$-primary over
$A\cup \bar a$. Hence there is an elementary embedding $f_{\bar b}:
M_{\bar a}\rightarrow M_{\bar b}, f_{\bar b}|_{(A\cup \bar a)}=id$.
Similarly, there is an elementary embedding $f_{\bar c}:
M_{\bar a}\rightarrow M_{\bar c}, f_{\bar c}|_{(A\cup \bar a)}=id$.
By the previous theorem, there are elementary mappings $g_{\bar
b},g_{\bar c}$ and a model $M$ with $g_{\bar b}:M_{\bar b}\rightarrow
M,g_{\bar c}: M_{\bar c}\rightarrow M$ and $g_{\bar b}\circ f_{\bar b}
=g_{\bar c}\circ f_{\bar c}$.
So in particular $g_{\bar b}| (A\cup \bar
a)=g_{\bar c}| (A\cup \bar a)=id$ and $P^M\subseteq A$. Now
$a^\frown g_{\bar b}(\bar b)^\frown g_{\bar c}(\bar c)$ is as required.
\end{proof}


\red{
\begin{de}
\begin{enumerate}
\item
We say that a model $N$ is $\lam$-full over a set $A$ if: 

$N$ is saturated, $P^N \subseteq A$, and for every $B\subseteq N$, $|B|<\lam$, every $p\in S_*(A\cup B)$ is realized in $N$. 

If $\lam = |A|$, we omit it. 
\item
We say that a model $N$ is $\lambda$-homogenous for sequences  if  whenever
$\langle a_i:i<\alpha \rangle ,\langle b_i:i<\alpha\rangle , \alpha
<\lambda$, realize the same
type in $N$, then for every $a_\alpha\in N$ there is $b_\alpha \in N$
such that $\langle a_i:i\leq \alpha\rangle , \langle \red{b_i}:i\leq \alpha\rangle $ realize the same
type in $N$.
\end{enumerate}
\end{de}
}

\red{
\begin{remark}
 If $N$ is $\lam$-full 
over $A$, then $(N,a)_{a\in A}$ is $\lambda$-homogenous for sequences.
\end{remark}
\begin{proof}
	Note that $P^N \subseteq A$, so for every sequence $\langle {b_i}:i\leq \alpha\rangle $, the type $tp(b_\al/A\cup\{ {b_i}:i<\alpha\})$ is in $S_*(A \cup \{ {b_i}:i<\alpha\})$.
\end{proof}
}

\red{

\begin{lem}\label{bigexists}
	Let $A$ is stable, $P^A$ $\lambda$-saturated with $\lambda
=|A|=\lambda^{<\lambda}>|T|$, then there is $M$ such that:
\begin{itemize}
\item[(i)]
$P^M\subseteq A$; and moreover
\item[(ii)] Every $p\in S_*(A)$ is realized.
\item[(iii)] $M$ is $\lambda$-saturated of cardinality $\lambda$.
\end{itemize}

\end{lem}
\begin{proof}
	Let $\lseq{p}{i}{\lam}$ list $S_*(A)$ (note: $A$ is stable, $\lam = \lam^{|T|}$). 
	
	By induction on $i\le \lambda$ choose $A_i$ increasing
continuously with $A_0=A$ and $|A_{i+1}\setminus A_i|<\lambda$, such that $A_i$ is complete, and $A_{i+1}$ realizes $p_i$. Use Amalgamation over $A$ (Theorem \ref{amalgamation}) to amalgamate $A_i$ and $a_i \models p_i$ at successor stages (recall that $\lam = \lam^{<\lam}$) and Fact \ref{7}(ii) for limit stages. 


Since $A_\lam$ is complete, $P\cap A_\lam = P\cap A$ saturated, by Proposition \ref{8} there is $M$ as required (since $\lam = \lam^{<\lam}$, $M$ is also saturated). 
	
\end{proof}

}

\red{
\begin{co}\label{fullexists}
	If $A$ is stable and $\lambda$-saturated with $\lambda
=|A|=\lambda^{<\lambda}>|T|$, there is $M$ of cardinality $\lambda$ which is full over $A$.
\end{co}
\begin{proof}
	Let $M_0$ be as in the previous Lemma. Now construct $M_i$ increasing 
	 (for $i<\lam$) such that $M_i$ satisfies the requirements (i) 
	-- (iii) of  the Lemma with $A$ there 
	replaced with  $\bigcup_{j<i}M_j$ (note that all models are stable, and $P^{M_i} = P^{M_0} = P\cap A$). 
	
	Clearly $M_\lam$ is as required (note that $\lam$ is regular). 
\end{proof}
}


\begin{co} If $A$ is stable and $\lambda$-saturated with $\lambda
=|A|=\lambda^{<\lambda}>|T|$, there is $M$ such that:
\begin{itemize}
\item[(i)] $M$ is $\lambda$-saturated of cardinality $\lambda$ with
$P^M\subseteq A$; and moreover
\item[(ii)] $(M,a)_{a\in A}$ is $\lambda$-homogenous for sequences and
every $p\in S_*(A)$ is realized.
\end{itemize}
\end{co}

%
%

\smallskip

\red{

Our next goal is a Symmetry Lemma for stationarizations over a model. 

We begin by showing that every ``Morley sequence'' (that is, a sequence of stationarizations of a given type) has a certain weak convergence property (which may remind the reader of the behaviour of indiscernible sequences in dependent theories). After having proved symmetry, we will conclude true convergence (since we will know that every such sequence is in fact an indiscernible set). However, we need the weak convergence property for the proof of symmetry, hence we deal with it first.

\begin{lem}\label{weakconvergence}\underline{(Weak Convergence over stable sets).} 
%
	Let $\seq{A_i\colon i\le \mu}$ be a sequence of stable sets 
	increasing 
	continuously, $A\subseteq_t A_i$, 
	$\bar a_i \subseteq A_{i+1}$ and $tp(\bar a_i/A_i)$ is
	the stationarization of $tp(\bar a_0/A)$. 	
	Let $\bar c\in A_\mu$ and $\psi (\bar x,\bar z, \bar w)$ a formula, $\theta (\bar z,\bar x,\bar w):=\psi(\bar x,\bar z, \bar w)$.
	Let $\Delta _2$ be finite such that  $n_\theta:=R_{A_{\mu}}(\bar x=\bar x,\{\theta\},\Delta_2,2)<\omega$ ($A_\mu$ is stable, so such $\Delta_2$, $n_\theta$ exist).
	
	Then there are $n\leq n_\theta$, $0=i_0<i_1<\dots <i_n=\lambda$,
	and $p_0(\x),\dots ,p_{n-1}(\x)$ such that for all $m<n, i_m<i<i_{m+1}$ implies $tp_\psi (\bar a_i/\bar c\cup A)=p_m$.
	

\end{lem}
\begin{proof}
By Fact \ref{rank} (i), $\langle R_{A_\mu}(tp(\bar c/A_\alpha ),\{\theta\},\Delta
_2,2):\alpha <\mu\rangle $ is a non-increasing sequence of natural
numbers $\leq n_\theta$. 


So there are $n\leq n_\theta, 0=i_0<\dots
<i_n=\mu$ such that

$$i_l\leq \alpha\leq \beta<i_{l+1} \Rightarrow R_{A_\mu}(tp_\theta(\bar
c/M_\alpha),\{\theta\},\Delta _2,2)=R_{A_\mu}(tp_\theta(\bar
c/M_\beta),\{\theta\},\Delta _2,2).$$

Again by Fact \ref{rank} {\color{black}(more specifically, by the proof of Corollary \ref{14} -- the nature of the defining scheme)}, 
for each $l$ there is $\bar d_l\in A_{i_l}$
and $\Psi _\theta(\bar x,\bar w,\bar d_l)$ which defines
$tp_\theta(\bar c/A_{i_l})$ so that $\Psi _\theta(\bar x,\bar
w,\bar d_l)$ actually defines $tp_\theta (\bar c/A_\alpha)$ for $i_l\leq \alpha
<i_{l+1}$. But if $i_l\leq \alpha <\beta <i_{l+1}$, then $tp(\bar
a_\alpha/A_{i_l})=tp(\bar a_\beta/A_{i_l})$,
hence for every $\bar m \subseteq A$, $\models \Psi _\theta(\bar a_\alpha,\bar m,\bar d_l)\equiv\Psi
_\theta(\bar a_\beta,\bar m,\bar d_l)$, hence $\models\theta(\bar
a_\alpha,\bar m,\bar c)\equiv\theta(\bar a_\beta,\bar m,\bar c)$, as required.

\end{proof}

}

\begin{theorem}\label{25}\label{symmetry}\underline{(The symmetry theorem for models)}. If $tp(\bar a\bar
b/M)\in S_*(M)$ where $tp(\bar b/M\cup\bar a)$ is the stationarization
of $tp(\bar b/M)$, then $tp(\bar a/M\cup\bar b)$ is the
stationarization of $tp(\bar a/M)$.
\end{theorem}
\begin{proof}
Assume we have a counterexample. Let $\lambda =\lambda^{<\lambda}\geq
|M|+|T|^+$. Without loss of generality $M$ is saturated of cardinality
$\lambda$. We define $\bar a_i,\bar b_i,M_i$ by induction on
$i<\lambda$ such that $P^{M_i}=P^{M_0}, M_{i+1}$ is
$\lambda$-saturated of power $\lambda$, $M_i$ increasing continuously,
$\bar a_i\bar b_i\subseteq M_{i+1}$ and $tp(\bar a_i\bar b_i/M_i)$ is
the stationarization of $tp(\bar a\bar b/M)$. This is straightforward.
Let $M_\lambda=\bigcup_{i<\lambda} M_i$.

\red{Since all models are stable, we are clearly in the situation of Lemma \ref{weakconvergence}. Note that $P^{M_\lam} = 
P^M \subseteq M$, hence by conclusion of the Lemma we get: }

For every $\bar c\in M_\lambda$ and $\psi (\bar x,\bar y,\bar z)$ there
are $n\leq n_\theta$ and $0=i_0<i_1<\dots <i_n=\lambda$ and $p_0,\dots
,p_{n-1}$ such that for all $m<n, i_m<i<i_{m+1}$ implies $tp_\psi (\bar
a_i\bar b_i/\bar c\cup P^M)=p_m$.

From the uniqueness \red{and definability} of stationarizations it follows that $\langle \bar a_i\bar
b_i:i<\lambda \rangle $ is indiscernible over $M_0$. That is, if $i_0<i_1<\dots
<i_n<\lambda$ and $j_0<\dots <j_n<\lambda$ then $tp(\bar a_{i_0}\bar
b_{i_0}\dots \bar a_{i_n}\bar b_{i_n}/M_0)=tp(\bar a_{j_0}\bar
b_{j_0}\dots \bar a_{j_n}\bar b_{j_n}/M_0)$.

Now let $R=\{\bar a_i\bar b_i:i<\lambda\}$ and let $<$ be the order on $R$
defined so that $\bar a'\bar b'<\bar a"\bar b"$ iff there are $i<j$ with
$\bar a'\bar b'=\bar a_i\bar b_i$ and $\bar a"\bar b"=\bar a_j\bar b_j$.
The model $(M_\lambda ,R,<)$ has a saturated extension, $(M^*,R^*,<^*)$
of power $\lambda$. So for some linear order $I$, $R^*=\{\bar a_t\bar
b_t:t\in I\}$ where $\bar a_s\bar b_s<\bar a_t\bar b_t$ whenever
$I\models s<t$.

Note that $n$ depends on $\psi$ and $\bar c$, but the bound $n_\theta$
depends on $\psi$ only. Therefore the following is true:

\begin{fact}\label{factsymmetry}
For every $\bar c\in M^*$ and $\psi(\bar x,\bar y,\bar z)$ there are
$n<\red{n_\theta <}\omega$ and $t_0<\dots <t_n$ where $t_0\in I$ is the first element,
and $p_0,\dots p_n$ so that $t_l<t<t_{l+1}$ or $t_l<t, l=n$ implies
$tp_\red{\psi}(\bar a_t\bar b_t/P^{M^*}\cup\bar c)=p_l$.
\end{fact}

Now we shall finish proving the Symmetry Lemma: Since $I$ is a
$\lambda$-saturated linear order of power $\lambda$, it has $2^\lambda$
Dedekind cuts $\{ (I_\alpha ,J_\alpha ):\alpha <2^\lambda\}$.
Let $p_\alpha =\{\psi(\bar x,\bar y,\bar d);\bar d\in M^*$ and for some
$s_\alpha\in I_\alpha,t_\alpha\in J_\alpha$, if $v\in I$ and $s_\alpha
<v<t_\alpha$ then $\models\psi(\bar a_v,\bar b_v,\bar d)\}$.

By Fact \ref{factsymmetry}, $p_\alpha$ is a complete type over $M^*$. In
fact, $p_\alpha\in S_*(M^*)$: \red{indeed,if $[\exists \z \in P \psi(\z,\bar m,\x,\y)] \in p_\al$ (where $\bar m \subseteq M^*$), then (by definition of $p_\al$) there exist $s_\al,t_\al$ such that for all $v\in (s_\al,t_\al)$, we have $\models \exists \z \in P \psi(\z,\bar m,\a_v,\b_v)$; hence $M^*$ satisfies this sentence as well, so there exists such $\d \subseteq P^{M^*}$. A priori perhaps $d$ depends on $v$; however, recall (Fact \ref{factsymmetry}) that for some $s'_\al,t'_\al$ all $a_v,b_v$ (for $v \in (s'_\al,t'_\al)$) have the same $\psi$-type over $\bar m \cup P^M*$, so choosing $d$ for any such $v$, we get that for all $v \in (s'_\al,t'_\al)$ the formula $\psi(\d,\bar m,\a_v,\b_v)$ holds, hence by the definition $\psi(\d,\bar m,\x, \y) \in p_\alpha$. Now by Observation \ref{8.5}, $p_\alpha \in S_*(M^*)$, as required. }

If $i<j, tp(\bar b_j/M\cup \bar
a_i)\subseteq tp(\bar b_j/M_j)$ is a stationarization of $tp(\bar b/M)$.
By the uniqueness of stationarizations and the assumption that $tp(\bar
b/M\cup\bar a)$ is the stationarization of $tp(\bar b/M)$ it follows
that $tp(\bar b_j\bar a_i/M)=tp(\bar b\bar a/M)$.

Similarly $tp(\bar a_1/M\cup \bar b_0)$ is the stationarization of
$tp(\bar a_1/M)$ and hence $\not= tp(\bar a_0/M\cup\bar b_0)$ as we
assumed that $\bar a\bar b$ form a counter-example to symmetry. So for
some $\bar e\in M$ and $\theta$ we have $\models\theta(\bar a,\bar b,\bar
e)\wedge\neg\theta(\bar a_1,\bar b_0,\bar e)$. \red {Therefore we get: 

$$ j \ge i \then \models\theta(\a_ib_j); \;\;\;\;\;\;\;\;\;\;\;  j < i \then \models\neg\theta(\a_ib_j) $$

So $\theta(\bar x,\bar
b_t)\in p_\alpha (\bar x,\bar y)$ if and only if $t\in J_\alpha$. Hence
if $\alpha\not=\beta$, then $p_\alpha\not= p_\beta$. }So we have too many
types in $S_*(M^*)$, contradicting stability \red{of $M^*$}.
\end{proof}

\begin{de}
\begin{itemize}
\item[(i)] We say that $I=\{a_\alpha:\alpha <\alpha ^*\}$ is {\it
convergent} over a set $A$ if it is an infinite indiscernible set such that for every
$\psi$ and $\bar d \in A$ there is some $n_\psi$ such that the type
$tp_\psi(\bar a_\alpha/P^{\cal C}\cup \bar d)$ is the same for all
but $\leq n_\psi$ many $\alpha$'s.
\item[(ii)] For any such $I$ let the {\sl average} of $I$ over $A$ be
$Av(I/A)=\{\phi(\bar x,\bar b):\bar b\in A, (\exists
^\infty\alpha)\phi(\bar a_\alpha,\bar b)\}$.
\end{itemize}
\end{de}

\begin{co}\label{28}\label{average} If $\alpha\geq\omega$ and for $i\leq\alpha, tp(\bar
a_i/M\cup\bigcup_{j<i}\bar a_j)$ is a stationarization of $tp(\bar
a/M)$, then $\{\bar a_i:i<\alpha\}$ is an indiscernible set over $M$.
That is, if $i_1,\dots ,i_n<\alpha$  are distinct, then $tp(\bar
a_{i_1}\dots\bar a_{i_n}/M)=tp(\bar a_1\dots\bar a_n/M)$. In addition, if
$I=\{\bar a^\prime_\alpha :\alpha <\alpha _0\}$ is an indiscernible set
with the same type over $M$, it is convergent over any $M'\succ M$, and for
$M^\prime\succ M, Av(I/M^\prime)$ is the stationarization  of $tp(\bar
a_0/M)$ over $M'$.
\end{co}
\red{
\begin{proof}
 	 It follows from Symmetry (Theorem \ref{25}) that any such sequence is an indiscernible set over $M$. Now a standard argument shows that for indiscernible sets weak convergence (Lemma \ref{weakconvergence}) implies convergence.
\end{proof}

We can therefore conclude:

\begin{co}
	The global stationarization of a type orthogonal to $P$ over a model $p\in S_*(M)$ is a generically stable type, as defined in \cite{PT,GOU}.
\end{co}

\begin{proof}
	First note that since $M$ is a model, $p$ has a (unique) global stationarization $p^*$, which is definable (hence invariant) over $M$. By the previous Corollary, every Morley sequence in $p^*$ 
	is convergent over $\cC$, in particular, $p^*$ is generically stable (as defined in Definition 1 of \cite{PT}). 
\end{proof}
}

One can  now use some basic machinery of generic stability to throw more light on *-types over models and the concept of stationarization, for example:

\begin{co}\label{co:nonforking}
	Let $M$ be a model and $p\in S_*(M)$. Then the unique global stationarization of $p$ is also the unique global 
	non-forking extension of $p$. 
\end{co}
\begin{proof}
By Proposition 1 in \cite{PT}.
\end{proof}

%

\bibliography{common.bib}

\def\cprime{$'$}
\begin{thebibliography}{GOU13}

\bibitem[Cha20]{Chatzidakis2020RemarksAT}
Zo{\'e} Chatzidakis.
\newblock Remarks around the non-existence of difference-closure.
\newblock {\em arXiv: Logic}, 2020.

\bibitem[GOU13]{GOU}
Dar\'{i}o Garc\'{i}a, Alf Onshuus, and Alexander Usvyatsov.
\newblock Generic stability, forking, and thorn-forking.
\newblock {\em Trans. Amer. Math. Soc.}, 365(1):1--22, 2013.

\bibitem[Hen14]{Hen-thesis}
Robert~S. Henderson.
\newblock Independence in exponential fields.
\newblock {\em Doctoral Thesis}, University of East Anglia, 2014.

\bibitem[Hod99]{Hod-cat1}
Wilfrid Hodges.
\newblock Relative categoricity in abelian groups.
\newblock In {\em Models and computability ({L}eeds, 1997)}, volume 259 of {\em
  London Math. Soc. Lecture Note Ser.}, pages 157--168. Cambridge Univ. Press,
  Cambridge, 1999.

\bibitem[Hod02]{Hod-cat3}
Wilfrid Hodges.
\newblock Relative categoricity in linear orderings.
\newblock In {\em Logic and algebra}, volume 302 of {\em Contemp. Math.}, pages
  235--248. Amer. Math. Soc., Providence, RI, 2002.

\bibitem[HS90]{HaSh323}
Bradd Hart and Saharon Shelah.
\newblock {Categoricity over $P$ for first order $T$ or categoricity for
  $\phi\in{\rm L}_ {\omega_ 1\omega}$ can stop at $\aleph_ k$ while holding for
  $\aleph_ 0,\cdots,\aleph_ {k-1}$}.
\newblock {\em Israel Journal of Mathematics}, 70:219--235, 1990.
\newblock arxiv:math/9201240.

\bibitem[HY09]{Hod-cat2}
Wilfrid Hodges and Anatoly Yakovlev.
\newblock Relative categoricity in abelian groups. {II}.
\newblock {\em Ann. Pure Appl. Logic}, 158(3):203--231, 2009.

\bibitem[KZ14]{KiZil}
Jonathan Kirby and Boris Zilber.
\newblock Exponentially closed fields and the conjecture on intersections with
  tori.
\newblock {\em Ann. Pure Appl. Logic}, 165(11):1680--1706, 2014.

\bibitem[Pil83]{Pil-cat}
Anand Pillay.
\newblock {$\aleph _{0}$}-categoricity over a predicate.
\newblock {\em Notre Dame J. Formal Logic}, 24(4):527--536, 1983.

\bibitem[PS85]{PiSh130}
Anand Pillay and Saharon Shelah.
\newblock Classification theory over a predicate. {I}.
\newblock {\em Notre Dame J. Formal Logic}, 26(4):361--376, 1985.

\bibitem[PT11]{PT}
Anand Pillay and Predrag Tanovi{\'c}.
\newblock Generic stability, regularity, and quasiminimality.
\newblock In {\em Models, logics, and higher-dimensional categories}, volume~53
  of {\em CRM Proc. Lecture Notes}, pages 189--211. Amer. Math. Soc.,
  Providence, RI, 2011.

\bibitem[She86]{Sh234}
Saharon Shelah.
\newblock Classification over a predicate. {II}.
\newblock In {\em Around classification theory of models}, volume 1182 of {\em
  Lecture Notes in Math.}, pages 47--90. Springer, Berlin, 1986.

\bibitem[She90]{Sh:c}
S.~Shelah.
\newblock {\em Classification theory and the number of nonisomorphic models},
  volume~92 of {\em Studies in Logic and the Foundations of Mathematics}.
\newblock North-Holland Publishing Co., Amsterdam, second edition, 1990.

\bibitem[Zil05]{Zil-field2}
B.~Zilber.
\newblock Pseudo-exponentiation on algebraically closed fields of
  characteristic zero.
\newblock {\em Ann. Pure Appl. Logic}, 132(1):67--95, 2005.

\end{thebibliography}
\bibliographystyle{alpha}

\end{document}